\documentclass[11pt,reqno]{amsart}
\usepackage{bm,url,amsmath,amssymb,amsthm,graphicx}

\usepackage[top=1in, bottom=1in, left=1in, right=1in]{geometry}
\usepackage{setspace}
\usepackage[percent]{overpic}
\usepackage{mathrsfs}
\allowdisplaybreaks

\newcommand{\Q}{\mathbb{Q}}
\newcommand{\R}{\mathbb{R}}
\newcommand{\Z}{\mathbb{Z}}
\newcommand{\N}{\mathbb{N}}
\newcommand{\SL}{\mathrm{SL}}
\newcommand{\PSL}{\mathrm{PSL}}
\newcommand{\FF}{\mathcal F}
\newcommand{\wFF}{\widetilde{\FF}}
\newcommand{\rar}{\rightarrow}
\newcommand{\ol}{\overline}
\newcommand{\bsl}{\backslash}
\newcommand{\Supp}{\mathop{\rm Supp}}

\newtheorem{theorem}{Theorem}%[section]
\newtheorem{corollary}{Corollary}
\newtheorem{proposition}{Proposition}
\newtheorem{lemma}{Lemma}

\theoremstyle{definition}
\newtheorem{remark}{Remark}

\numberwithin{equation}{section}
\title[Poincar\'e sections for the horocycle flow in covers of $\SL(2,\R)/\SL(2,\Z)$]{Poincar\'e sections for the horocycle flow in covers of $\SL(2,\R)/\SL(2,\Z)$ and applications to Farey fraction statistics}
\author{Byron Heersink}

\address{Department of Mathematics, University of Illinois Urbana-Champaign, Urbana, IL 61801}

\email{heersin2@illinois.edu}

\begin{document}

\begin{abstract}
For a given finite index subgroup $H\subseteq\SL(2,\Z)$, we use a process developed by Fisher and Schmidt to lift a Poincar\'e section of the horocycle flow on $\SL(2,\R)/\SL(2,\Z)$ found by Athreya and Cheung to the finite cover $\SL(2,\R)/H$ of $\SL(2,\R)/\SL(2,\Z)$. We then use the properties of this section to prove the existence of the limiting gap distribution of various subsets of Farey fractions. Additionally, to each of these subsets of fractions, we extend solutions by Xiong and Zaharescu, and independently Boca, to a Diophantine approximation problem of Erd\H{o}s, Sz\"usz, and Tur\'an.
\end{abstract}

\maketitle

\section{Introduction}\label{intro}

The significant work of Elkies and McMullen \cite{EM1} and of Marklof and Str\"ombergsson \cite{MS} has demonstrated that ergodic properties of homogeneous flows can provide a very powerful device in the study of the limiting gap distributions of certain sequences of arithmetic origin. Recently, Athreya and Cheung \cite{AC} realized the horocycle flow on $\SL(2,\R)/\SL(2,\Z)$ as a suspension flow over the BCZ map introduced by Boca, Cobeli, and Zaharescu \cite{BCZ} in their study of statistical properties of Farey fractions. Athreya and Cheung used this connection to rederive the limiting gap distribution and other properties of Farey fractions. The process used in \cite{AC} to obtain these results was later generalized in \cite{A} to explain the gap distributions of various different sequences. Also, recent work of Fisher and Schmidt \cite{FS} was concerned with lifting Poincar\'e sections of the geodesic flow on the unit tangent bundle of the modular surface to a finite cover, with the primary aim of obtaining statistical properties of continued fractions.

In this paper, we use the process in \cite{FS} to explicitly lift, for every finite index subgroup $H\subseteq\SL(2,\Z)$, the section of the horocycle flow discovered in \cite{AC} to the cover $\SL(2,\R)/H$ of $\SL(2,\R)/\SL(2,\Z)$. As one application, we prove in Theorem \ref{T1} the existence of the limiting gap measure of certain subsets of Farey fractions following the ideas in \cite{AC} and the more general framework of \cite[Theorem 2.5]{A}. A given subset we consider is determined, as described below, by a finite index subgroup $H\subseteq\SL(2,\Z)$, and corresponds to a Poincar\'e section of $\SL(2,\Z)/H$ obtained by intersecting the lift of the section in \cite{AC} with certain sheets of the cover $\SL(2,\R)/H\rar\SL(2,\R)/\SL(2,\Z)$. As a second application, we prove Theorem \ref{T2}, which establishes, for each of the aforementioned subsets of Farey fractions, the existence of the limiting Lebesgue measure of the real numbers in $[0,1]$ that are in some sense well-approximated by elements in the subset. This solves an analogue of a Diophantine approximation problem posed by Erd\H{o}s, Sz\"usz, and Tur\'an \cite{EST}.

Recall that the Farey sequence of order $Q$ is the set $\FF(Q)$ of fractions $\frac{a}{q}\in[0,1]$ such that $(a,q)=1$ and $q\leq Q$. Various properties regarding the spacing statistics of the increasing sequence $(\FF(Q))$ of subsets of $[0,1]$ have been studied (see, e.g., \cite{Hall}, \cite{ABCZ}, \cite{BZ}). Additionally, certain subsets of Farey sequences have been considered. For instance, if $\FF_{Q,d}\subseteq\FF(Q)$ is the set of fractions $\frac{a}{q}$ with $(q,d)=1$ and $\wFF_{Q,\ell}\subseteq\FF(Q)$ is the set of fractions $\frac{a}{q}$ with $\ell\nmid a$, then the number of pairs $(\frac{a}{q},\frac{a'}{q'})$ of consecutive fractions in $\FF_{Q,d}$ with fixed $a'q-aq'=k$ has been estimated by Badziahin and Haynes \cite{BH}, the pair correlation function of the sequence $(\FF_{Q,d_Q})$ was shown to exist by Xiong and Zaharescu \cite{XZ1} where $d_Q$ varies with $Q$ subject to the constraints $d_{Q_1}\mid d_{Q_2}$ as $Q_1<Q_2$ and $d_Q\ll Q^{\log \log Q/4}$, and the limiting gap distribution measure for the sequences $(\FF_{Q,d})$ and $(\wFF_{Q,\ell})$ were shown to exist for fixed $d$ and $\ell$ by Boca, Spiegelhalter, and the author \cite{BHS}.

For a given finite subset $A=\{x_0\leq x_1\leq\cdots\leq x_N\}$ of $[0,1]$, we define the \textit{gap distribution measure} of $A$ to be the probability measure $\nu_A$ on $[0,\infty)$ such that \[\nu_A[0,\xi]=\frac{1}{N}\#\{j\in[1,N]:N(x_j-x_{j-1})\leq\xi(x_N-x_0)\}, \quad \xi\geq0.\] For a sequence $(A_n)$ of finite subsets of $[0,1]$, we call the weak limit of $(\nu_{A_n})$, if it exists, the \textit{limiting gap measure} of $(A_n)$.

Let $G=\SL(2,\R)$ and $\Gamma=\SL(2,\Z)$. Our first main result is the following:

\begin{theorem}\label{T1}
Let $H$ be a finite index subgroup of $\Gamma$ and $M\subseteq\Gamma/H$ be a nonempty subset, closed under left multiplication by $\left(\begin{smallmatrix}1&-1\\0&1\end{smallmatrix}\right)$. Also, for $Q\in\N$, let $\FF_M(Q)\subseteq\FF(Q)$ be the set of fractions $\frac{a}{q}$ such that
\[\left(\begin{array}{cc}q'&a'\\-q&-a\end{array}\right)H\in M,\]
where $\frac{a'}{q'}$ is the successor of $\frac{a}{q}$ in $\FF(Q)$. Then the sequence $(\FF_M(Q))$ becomes equidistributed in $[0,1]$ as $Q\rar\infty$. Furthermore, if $I\subseteq[0,1]$ is a given subinterval and $\FF_{I,M}(Q)=\FF_M(Q)\cap I$, then the limiting gap measure $\nu_{I,M}$ of $(\FF_{I,M}(Q))$ exists and has a continuous and piecewise real-analytic density.
\end{theorem}

We include the hypothesis that $M$ is closed under left multiplication by $\left(\begin{smallmatrix}1&-1\\0&1\end{smallmatrix}\right)$ in order to ensure that $(\FF_M(Q))$ is an increasing sequence of sets. Indeed, let $\frac{a}{q}\in\FF_M(Q)$ so that
\[\left(\begin{array}{cc}q'&a'\\-q&-a\end{array}\right)H\in M,\]
where $\frac{a'}{q'}$ is the successor of $\frac{a}{q}$ in $\FF(Q)$. If $Q'\geq Q$, then by the mediant property of Farey fractions, the successor of $\frac{a}{q}$ in $\FF(Q')$ is equal to $\frac{na+a'}{nq+q'}$ for some $n\geq0$. We then have
\[\left(\begin{array}{cc}nq+q'&na+a'\\-q&-a\end{array}\right)H=\left(\begin{array}{cc}1&-1\\0&1\end{array}\right)^n\left(\begin{array}{cc}q'&a'\\-q&-a\end{array}\right)H\in M,\]
implying that $\frac{a}{q}\in\FF_M(Q')$, and hence $(\FF_M(Q))$ is increasing.

Applying Theorem \ref{T1} with $H=\Gamma(m)$, where $m$ is a positive integer and $\Gamma(m)$ is the congruence subgroup
\[\left\{\left(\begin{array}{cc}a&b\\c&d\end{array}\right)\equiv\left(\begin{array}{cc}1&0\\0&1\end{array}\right)\bmod{m}:a,b,c,d\in\Z,ad-bc=1\right\}\]
of $\Gamma$, and with
\[M=\left\{\left(\begin{array}{cc}n_4&n_3\\-n_2&-n_1\end{array}\right)H\in\Gamma/H:(n_1,n_2)\bmod{m}\in A\right\}\]
where $A\subseteq(\Z/m\Z)^2$ is such that $M$ is nonempty, i.e., there is some $(n_1,n_2)\in A$ such that $(n_1,n_2,m)=1$, we have the following result:

\begin{corollary}\label{C1}
Let $A\subseteq(\Z/m\Z)^2$ contain some $(n_1,n_2)$ such that $(n_1,n_2,m)=1$, and let $I\subseteq[0,1]$ be a subinterval. Then for $Q\in\N$, let $\FF_{m,A}(Q)$ be the set of fractions $\frac{a}{q}\in\FF(Q)$ such that $(a,q)\equiv(n_1,n_2)\bmod{m}$ for some $(n_1,n_2)\in A$, and let $\FF_{I,m,A}(Q)=\FF_{m,A}(Q)\cap I$. Then $(\FF_{m,A}(Q))$ becomes equidistributed in $[0,1]$ as $Q\rar\infty$, and the limiting gap measure $\nu_{I,m,A}$ of $(\FF_{I,m,A}(Q))$ exists and has a continuous and piecewise real-analytic density.
\end{corollary}

Corollary \ref{C1} includes the existence of the limiting gap measures of $(\FF_{Q,d})$ and $(\wFF_{Q,\ell})$ proven in \cite{BHS} as special cases since $\FF_{Q,d}=\FF_{d,A}(Q)$, where $A\subseteq(\Z/m\Z)^2$ is the subset consisting of all pairs $(n_1,n_2)$ such that $(n_2,d)=1$, and $\wFF_{Q,\ell}=\FF_{\ell,A'}$, where $A'\subseteq(\Z/m\Z)^2$ is the subset having all pairs $(n_1,n_2)$ with $n_1\not\equiv0\bmod{\ell}$. Congruence subgroups also appear in the study \cite[Corollary 2.7]{MS} of the related problem of proving the existence of the limiting gap measure for the angles of visible points in $\Z^2$ with respect to an observer at a rational point. See \cite{ACL} for a similar application of the ergodic properties of the horocycle flow on $G/H$, in the case where $H$ is the Hecke $(2,5,\infty)$ triangle group, to the computation of the limiting gap measure of slopes on the golden L.

In \cite{EST}, Erd\H{o}s, Sz\"usz, and Tur\'an introduced the Diophantine problem concerning the sets
\begin{align*}
S(n,\alpha,c)=\left\{\xi\in[0,1]:\mbox{there exists }a,q\in\Z\mbox{ such that }(a,q)=1,n\leq q\leq nc,|q\xi-a|\leq\frac{\alpha}{q}\right\}%\\
%\quad\quad\quad\quad(a,q)=1,n\leq q\leq nc,|q\xi-a|\leq\alpha/q\},
\end{align*}
where $n\in\N$, $\alpha>0$, and $c\geq1$. They posed the problem of deciding the existence of the limit of the Lebesgue measures
\[\lim_{n\rar\infty}\lambda(S(n,\alpha,c))\]
for all $\alpha$ and $c$. The limit was shown to exist by Kesten and S\'os \cite{KS}; and later, Xiong and Zaharescu \cite{XZ2}, and independently Boca \cite{Boca}, obtained formulas for how to calculate the limit explicitly. This type of problem was recently investigated in higher dimensions as well as the setting of translation surfaces by Athreya and Ghosh \cite{AG}. Our second main result is the following theorem, which establishes the limiting measure for sets defined in the same way as $S(n,\alpha,c)$, with the restriction that the pairs $(a,q)$ are such that $\frac{a}{q}\in\FF_M(Q)$ for some $Q\in\N$.

\begin{theorem}\label{T2}
Let $\FF_M(Q)$ be the set of Farey fractions defined as in Theorem \ref{T1}. Then for $n\in\N$, $\alpha>0$, and $c\geq1$, let
\begin{align*}
S_M(n,\alpha,c)=\bigg\{\xi\in[0,1]:&\,\,\mbox{there exists }\frac{a}{q}\in\FF_M(\lfloor nc\rfloor)\mbox{ in lowest terms}\\
&\,\,\mbox{such that }q\geq n,|q\xi - a|\leq\frac{\alpha}{q}\bigg\}
\end{align*}
and for a given subinterval $I\subseteq[0,1]$, $S_{I,M}(n,\alpha,c)=I\cap S_M(n,\alpha,c)$. Then the limits
\[\lim_{n\rar\infty}\lambda(S_M(n,\alpha,c))\quad\mbox{and}\quad\lim_{n\rar\infty}\lambda(S_{I,M}(n,\alpha,c))\]
exist, and if
\[\lim_{n\rar\infty}\lambda(S_M(n,\alpha,c))=\varrho_M(\alpha,c), \quad \mbox{then} \quad \lim_{n\rar\infty}\lambda(S_{I,M}(n,\alpha,c))=|I|\varrho_M(\alpha,c).\]
\end{theorem}

Again, letting $H=\Gamma(m)$ and $M\subseteq\Gamma/H$ be defined as above, we obtain the following corollary:

\begin{corollary}\label{C2}
Let $A\subseteq(\Z/m\Z)^2$ contain some $(n_1,n_2)$ such that $(n_1,n_2,m)=1$. Then for $n\in\N$, $\alpha>0$, and $c\geq1$, let
\begin{align*}
S(n,\alpha,c,A)=\bigg\{\xi\in[0,1]:&\,\,\mbox{there exists }a,q\in\Z_{\geq0}\mbox{ such that }(a,q)=1,\\
&\,\,(a,q)\bmod{m}\in A,n\leq q\leq nc,|q\xi - a|\leq\frac{\alpha}{q}\bigg\}
\end{align*}
and for a given subinterval $I\subseteq[0,1]$, $S_I(n,\alpha,c,A)=I\cap S(n,\alpha,c,A)$. Then the limits
\[\lim_{n\rar\infty}\lambda(S(n,\alpha,c,A))\quad\mbox{and}\quad\lim_{n\rar\infty}\lambda(S_I(n,\alpha,c,A))\]
exist, and
\[\lim_{n\rar\infty}\lambda(S_I(n,\alpha,c,A))=|I|\lim_{n\rar\infty}\lambda(S(n,\alpha,c,A)).\]
\end{corollary}

In Section \ref{sec:2}, we review the work of Athreya and Cheung \cite{AC} in showing the horocycle flow as a suspension flow by giving a Poincar\'e section $\Omega'$ of the horocycle flow in which the first return map is the BCZ map. We also outline how $\Omega'$ relates to the gaps in Farey fractions and how Athreya and Cheung used the ergodic properties of the horocycle flow to prove the equidistribution of certain points in $\Omega'$ that we call Farey points, which in turn yields results about Farey fraction gaps. Then in Sections \ref{sec:3}--\ref{sec:6}, we prove Theorem \ref{T1} using the same process. We start in Section \ref{sec:3} by proving that $\bigcup_{Q\in\N}\FF_M(Q)$ is dense in $[0,1]$, and this involves proving an elementary lemma regarding representatives for cosets in $\Gamma/H$. In Section \ref{sec:4}, we use results in \cite{FS} to construct a Poincar\'e section $\Omega_M$ of the horocycle flow on $G/H$ analogous to $\Omega'$ that relates to the gaps in $(\FF_M(Q))$. In Section \ref{sec:5}, we prove some important properties of the first return time function of $\Omega_M$ which have an effect on the existence and properties of $\nu_{I,M}$ mentioned in Theorem \ref{T1}. Then in Section \ref{sec:6}, we prove the equidistribution of certain points in $\Omega_M$, analogous to the Farey points in $\Omega'$. From this equidistribution we can conclude the existence of the limiting gap measure of $(\FF_{I,M}(Q))$. In Section \ref{sec:7}, we examine a particular property that we call the repulsion gap of $(\FF_{I,M}(Q))$, which is the infimum of the support of $\nu_{I,M}$. In the particular case where $(\FF_{I,M}(Q))$ is the sequence of Farey fractions $\frac{a}{q}$ with $q\equiv1\bmod{m}$ for some fixed $m\in\N$, the repulsion gap is explicitly computed as 
\begin{equation}\label{e12}
\frac{3}{\pi^2m\prod_{p|m}(1-\frac{1}{p^2})}.
\end{equation}
Lastly, in Section \ref{sec:8}, we prove Theorem \ref{T2} also as a corollary to the work in Sections \ref{sec:3}--\ref{sec:6}.

\section{The BCZ map and the Poincar\'e section of Athreya and Cheung}\label{sec:2}

Let $\Omega\subseteq\R^2$ be the Farey triangle containing the points $(a,b)$ satisfying $0<a,b\leq1$ and $a+b>1$. The BCZ map $T:\Omega\rar\Omega$ is defined in \cite{BCZ} by
\[T(a,b)=\left(b,\left\lfloor\frac{1+a}{b}\right\rfloor b-a\right).\]
This map has relevance to Farey fractions as follows: Let 
\[\FF(Q)=\left\{\gamma_0=\frac{a_0}{q_0}=\frac{0}{1}<\gamma_1=\frac{a_1}{q_1}<\cdots<\gamma_{N(Q)}=\frac{a_{N(Q)}}{q_{N(Q)}}=\frac{1}{1}\right\}\]
with $(a_i,q_i)=1$. We then have $a_{i+2}=Ka_{i+1}-a_i$ and $q_{i+2}=Kq_{i+1}-q_i$, where $K=\lfloor\frac{Q+q_i}{q_{i+1}}\rfloor=\lfloor\frac{1+q_i/Q}{q_{i+1}/Q}\rfloor$. As a result, $\frac{q_{i+2}}{Q}=K\frac{q_{i+1}}{Q}-\frac{q_i}{Q}$, which is the second coordinate of $T(\frac{q_i}{Q},\frac{q_{i+1}}{Q}$). Hence we have the equality
\[T\left(\frac{q_i}{Q},\frac{q_{i+1}}{Q}\right)=\left(\frac{q_{i+1}}{Q},\frac{q_{i+2}}{Q}\right).\]
We therefore have a correspondence $\frac{a_i}{q_i}\leftrightarrow(\frac{q_i}{Q},\frac{q_{i+1}}{Q})$ between $\FF(Q)$ and the subset $\{(\frac{q_i}{Q},\frac{q_{i+1}}{Q}):\gamma_i\in\FF(Q)\}$, which we refer to as the set of Farey points, and $T$ maps a Farey point corresponding to a certain fraction to the Farey point corresponding to the succeeding fraction.

In \cite{AC}, the Farey triangle was viewed as a subset of $G/\Gamma$ by letting
\[P=\left\{p_{a,b}=\left(\begin{array}{cc}a&b\\0&a^{-1}\end{array}\right):(a,b)\in\Omega\right\}\]
and considering the set
\[\Omega'=P\Gamma/\Gamma=\{\Lambda_{a,b}=p_{a,b}\Gamma:(a,b)\in\Omega\},\]
which was found to be a Poincar\'e section for the horocycle flow, viewed as the action by left multiplication on $G/\Gamma$ of
\[N=\left\{h_s=\left(\begin{array}{cc} 1 & 0 \\ -s & 1\end{array}\right):s\in\R\right\}.\]
This means that for almost every $\Lambda\in G/\Gamma$ (with respect to the Haar measure), the set $\{s\in\R:h_s\Lambda\in\Omega'\}$ of times the orbit of $\Lambda$ under the horocycle flow meets $\Omega'$ is nonempty, countable, and discrete. The first return time function $R:\Omega'\rar\R$ defined by $R(\Lambda_{a,b})=\min\{s>0:h_s\Lambda_{a,b}\in\Omega'\}$ is $R(\Lambda_{a,b})=\frac{1}{ab}$, and the first return map $r:\Omega'\rar\Omega'$ defined by $r(\Lambda_{a,b})=h_{r(\Lambda_{a,b})}\Lambda_{a,b}$ is $r(\Lambda_{a,b})=\Lambda_{T(a,b)}$, where $T$ is the BCZ map. The last equality can be seen by the calculation
\begin{align*}
h_{R(\Lambda_{a,b})}\Lambda_{a,b}&=
\left(\begin{array}{cc}1&0\\-\frac{1}{ab}&1\end{array}\right)\left(\begin{array}{cc}a&b\\0&a^{-1}\end{array}\right)\Gamma=\left(\begin{array}{cc}a&b\\-b^{-1}&0\end{array}\right)\left(\begin{array}{cc}0&-1\\1&\lfloor\frac{1+a}{b}\rfloor\end{array}\right)\Gamma\\
&=\left(\begin{array}{cc}b&\lfloor\frac{1+a}{b}\rfloor b-a\\0&b^{-1}\end{array}\right)\Gamma=\Lambda_{T(a,b)}.
\end{align*}
Also, if we identify $G/\Gamma$ with the set $\{(a,b,s):(a,b)\in\Omega,0\leq s<\frac{1}{ab}\}$ via the correspondence $h_s\Lambda_{a,b}\leftrightarrow(a,b,s)$, then $d\mu_{G/\Gamma}=2\,da\,db\,ds$, where $\mu_{G/\Gamma}$ is the Haar measure on $G/\Gamma$ such that $\mu_{G/\Gamma}(G/\Gamma)=\frac{\pi^2}{3}$.

Letting $I\subseteq[0,1]$ be a subinterval, $\FF_I(Q)=\FF(Q)\cap I$, and $N_I(Q)=\#\FF_I(Q)$, we define on $\Omega'$ the measure \[\rho_{Q,I}=\frac{1}{N_I(Q)}\sum_{i:\gamma_i\in I}\delta_{r^i(\Lambda_{1,1/Q})}=\frac{1}{N_I(Q)}\sum_{i:\gamma_i\in I}\delta_{\Lambda_{q_i/Q,q_{i+1}/Q}}.\]
Notice that $R(\Lambda_{q_i/Q,q_{i+1}/Q})=\frac{Q^2}{q_iq_{i+1}}=Q^2(\gamma_{i+1}-\gamma_i)$, and as a result,
\begin{equation}
\frac{\#\{\gamma_i\in\FF_I(Q):\frac{3}{\pi^2}Q^2(\gamma_{i+1}-\gamma_i)\in[0,c]\}}{N_I(Q)}=\rho_{Q,I}\left(R^{-1}\left[0,\frac{\pi^2}{3}c\right]\right)\label{e21}
\end{equation}
for all $c\geq0$. Now we have $N_I(Q)\sim\frac{3}{\pi^2}|I|Q^2$, and if $\gamma_{I,l}$ and $\gamma_{I,g}$ are the least and greatest elements in $\FF_I(Q)$, respectively (we suppress the dependence on $Q$), then $\gamma_{I,g}-\gamma_{I,l}\rar|I|$ as $Q\rar\infty$. So the limit of the left side of \eqref{e21} as $Q\rar\infty$ is the measure of $[0,c]$ under the limiting gap measure of $(\FF_I(Q))$. So to show that the limiting gap measure of $(\FF_I(Q))$ exists, it suffices to prove that the limit of the right side of \eqref{e21} exists. To do so, Athreya and Cheung proved that the sequence $(\rho_{Q,I})$ of measures converges in the weak* topology to the measure $m$ on $\Omega'$ given by $dm=2\,da\,db$. They first noticed that if $\rho_{Q,I}^R$ is the measure on $G/\Gamma$ defined by $d\rho_{Q,I}^R=d\rho_{Q,I}\,ds$, where we are viewing $G/\Gamma$ as the set $\{(\Lambda_{a,b},s)\in\Omega'\times\R:0\leq s<\frac{1}{ab}\}$ by the correspondence $h_s\Lambda_{a,b}\leftrightarrow(\Lambda_{a,b},s)$, then $\rho_{Q,I}^R\rar\mu_{G/\Gamma}$ in the weak* topology. This convergence is a consequence of the equidistribution of closed horocycles in $G/\Gamma$ (see, e.g., \cite{Z}, \cite{Sa}, \cite{EM2}, \cite{Hj}, \cite{St}).

It then follows that if $\pi_{\Omega'}:G/\Gamma\rar\Omega'$ is the projection $(\Lambda_{a,b},s)\mapsto\Lambda_{a,b}$ (we are again viewing $G/\Gamma$ as $\{(\Lambda_{a,b},s)\in\Omega'\times\R:0\leq s<\frac{1}{ab}\}$), then \[\frac{1}{r}\pi_{\Omega'*}\rho_{Q,I}^R\rar\frac{1}{r}\pi_{\Omega'*}\mu_{G/\Gamma}\quad(Q\rar\infty)\] in the weak* topology. It is easy to see that $\rho_{Q,I}=\frac{1}{r}\pi_{\Omega'*}\rho_{Q,I}^R$ and $m=\frac{1}{r}\pi_{\Omega'*}\mu_{G/\Gamma}$, and so $\rho_{Q,I}\rar m$.

\begin{remark} The convergence $\rho_{Q,I}\rar m$ was proven in \cite{KZ} in the case $I=[0,1]$. This can actually be proven in an elementary way using M\"obius summation. For short subintervals $I=I(Q)\subseteq[0,1]$ with $|I(Q)|\gg Q^{-1/2+\epsilon}$, this convergence can be deduced using a corollary of the Weil bound for Kloosterman sums (see, e.g., \cite[Section 2]{BGZ}). For fixed $I$, this convergence also follows from \cite[Theorem 6]{M1}, in which the equidistribution of Farey points of arbitrary dimension was proven. See \cite{M2} for results regarding the spacing statistics of higher-dimensional Farey fractions. Using $\rho_{Q,I}\rar m$, Athreya and Cheung \cite{AC} not only explained the gap distribution of $(\FF_I(Q))$, but through finding appropriate functions $f:\Omega'\rar\R$ such that \[\lim_{Q\rar\infty}\int_{\Omega'}f\,d\rho_{Q,I}=\int_{\Omega'}f\,dm,\]
they were able to recast in this unified setting many previously known results about Farey fractions, including results on their $h$-spacings and indices.
\end{remark}

\section{The density of $\bigcup_{Q\in\N}\FF_M(Q)$ in $[0,1]$}\label{sec:3}

Throughout this section and Sections \ref{sec:4}--\ref{sec:6}, let $H\subseteq\Gamma$ be a subgroup of finite index, $M=\{m_1H,\ldots,m_kH\}\subseteq\Gamma/H$ be a nonempty subset closed under left multiplication by $\left(\begin{smallmatrix}1&-1\\0&1\end{smallmatrix}\right)$, and $I=[x_1,x_2]\subseteq[0,1]$ be a subinterval. We now set out to prove that the limiting gap measure $\nu_{I,M}$ of $(\FF_{I,M}(Q))$ exists and start in this section by proving that $\bigcup_{Q\in\N}\FF_M(Q)$ is dense in $[0,1]$. We first prove the following elementary lemma:

\begin{lemma}\label{L1}
Let $gH$ be any coset in $\Gamma/H$. Then there exist positive integers $a,b,c,d$ such that \[\left(\begin{array}{cc}a&b\\-c&-d\end{array}\right)\in gH.\]
\end{lemma}

\begin{proof}
First note that since $[\Gamma:H]<\infty$, there exists an integer $N\geq2$ such that \[U_N=\left(\begin{array}{cc}1&N\\0&1\end{array}\right),\quad L_N=\left(\begin{array}{cc}1&0\\N&1\end{array}\right)\in H.\] Let $A=\left(\begin{smallmatrix}a_0&b_0\\c_0&d_0\end{smallmatrix}\right)\in gH$. If $a_0>0$ and $b_0=0$, then $A=\left(\begin{smallmatrix}1&0\\c_0&1\end{smallmatrix}\right)$. Replacing $A$ by $AL_N^{-j}$ for a large enough $j$ replaces $c_0$ by a number less than $-1$, and so we can assume that $c_0<-1$. We then have $AU_N=\left(\begin{smallmatrix}1&N\\c_0&c_0N+1\end{smallmatrix}\right)$, which is a matrix of the desired form. So the proof is complete in this case.

So assume that $a_0\leq0$ or $b_0\neq0$. If $b_0=0$, we must have $a_0<0$, and if $a_0=0$ so that $b_0\neq0$, multiplying $A$ on the right by $L_N$ or $L_N^{-1}$ replaces $a_0$ by a negative number. So we can assume that $a_0<0$. Then multiplying $A$ on the right by $U_N^{-j}$ for a large enough $j$ replaces $b_0$ by a positive number, and so assume that $b_0>0$. Now since $a_0d_0-b_0c_0=1$, we clearly have $c_0d_0\leq0$. Suppose that $d_0=0$, implying that $A=\left(\begin{smallmatrix}a_0&1\\-1&0\end{smallmatrix}\right)$. Multiplying $A$ on the right by $L_N^jU_N$ yields $\left(\begin{smallmatrix}a_0+jN&N(a_0+jN)+1\\-1&-N\end{smallmatrix}\right)$. Choosing $j$ so that $a_0+jN>0$ yields a matrix of the desired form, and so the proof is complete in this case. If $c_0=0$, then $A=\left(\begin{smallmatrix}-1&b_0\\0&-1\end{smallmatrix}\right)$, and multiplying $A$ on the right by $L_N^{-1}$ yields $\left(\begin{smallmatrix}-1-b_0N&b_0\\N&-1\end{smallmatrix}\right)$.

Thus we have reduced the case where $c_0=0$ to the situation where $A=\left(\begin{smallmatrix}-a&b\\c&-d\end{smallmatrix}\right)$ with $a,b,c,d>0$, which we now consider. Let $\frac{\gamma}{\alpha}<\frac{\delta}{\beta}$ be fractions such that $\alpha\delta-\beta\gamma=1$ and $\frac{a}{b}<\frac{\gamma}{\alpha}$. Matrix multiplication reveals that any power of $B=\left(\begin{smallmatrix}\alpha&\beta\\\gamma&\delta\end{smallmatrix}\right)$ is of the form $\left(\begin{smallmatrix}\alpha'&\beta'\\\gamma'&\delta'\end{smallmatrix}\right)$ where $\frac{\gamma}{\alpha}\leq\frac{\gamma'}{\alpha'}<\frac{\delta'}{\beta'}\leq\frac{\delta}{\beta}$. So noting that $[\Gamma:H]<\infty$, we can replace $B$ by some power of $B$ that is in $H$. We then have $AB\in gH$ and
\[AB=\left(\begin{array}{cc}-a&b\\c&-d\end{array}\right)\left(\begin{array}{cc}\alpha&\beta\\\gamma&\delta\end{array}\right)=\left(\begin{array}{cc}-a\alpha+b\gamma&-a\beta+b\delta\\c\alpha-d\gamma&c\beta-d\delta\end{array}\right),\]
which is a matrix of the desired form since $\frac{c}{d}<\frac{a}{b}<\frac{\gamma}{\alpha}<\frac{\delta}{\beta}$.

The last case we need to consider is when $A=\left(\begin{smallmatrix}-a&b\\-c&d\end{smallmatrix}\right)$ with $a,b,c,d>0$. Noting that $\frac{a}{b}<\frac{c}{d}$ since $-ad+bc=1$, we can find fractions $\frac{\gamma}{\alpha}<\frac{\delta}{\beta}$ such that $\alpha\delta-\beta\gamma=1$ and $\frac{a}{b}<\frac{\gamma}{\alpha}<\frac{\delta}{\beta}<\frac{c}{d}$. As in the previous case, we may assume that $B=\left(\begin{smallmatrix}\alpha&\beta\\\gamma&\delta\end{smallmatrix}\right)\in H$. Then $AB\in gH$ and
\[AB=\left(\begin{array}{cc}-a&b\\-c&d\end{array}\right)\left(\begin{array}{cc}\alpha&\beta\\\gamma&\delta\end{array}\right)=\left(\begin{array}{cc}-a\alpha+b\gamma&-a\beta+b\delta\\-c\alpha+d\gamma&-c\beta+d\delta\end{array}\right)\]
is a matrix of the desired form. This completes the proof.
\end{proof}

So by Lemma \ref{L1}, there exist $a,b,c,d\in\N$ such that $\left(\begin{smallmatrix}a&b\\-c&-d\end{smallmatrix}\right)H\in M$. Multiplying $\left(\begin{smallmatrix}a&b\\-c&-d\end{smallmatrix}\right)$ on the right by $L_N$ yields a matrix of the form $\left(\begin{smallmatrix}q'&a'\\-q&-a\end{smallmatrix}\right)$, where $\frac{a}{q}<\frac{a'}{q'}$ are consecutive Farey fractions of some order. We have $\left(\begin{smallmatrix}q'&a'\\-q&-a\end{smallmatrix}\right)H\in M$, proving that $\bigcup_{Q\in\N}\FF_M(Q)$ is nonempty.

\begin{lemma}\label{L2}
There exists constants $Y>0$ and $Q_0>0$, depending only on the subgroup $H\subseteq G$, such that for any $Q\geq Q_0$ and $x\in[0,1]$,
\[\min_{\beta\in\FF_M(Q)}|x-\beta|\leq\frac{Y}{Q}.\]
\end{lemma}

\begin{proof}
 By Lemma 1, there exists a matrix of the form $\left(\begin{smallmatrix}a&b\\-c&-d\end{smallmatrix}\right)$, with $a,b,c,d>0$, in each coset $gH$ in $\Gamma/H$. By multiplying each on the right by a sufficiently large power of $\left(\begin{smallmatrix}1&-N\\0&1\end{smallmatrix}\right)$, we can say that each coset in $\Gamma/H$ also contains a matrix of the form $\left(\begin{smallmatrix}a&-b\\-c&d\end{smallmatrix}\right)$, with $a,b,c,d>0$. Let $k_H>0$ be an upper bound on the absolute values of entries of coset representatives in $\Gamma/H$ of the forms $\left(\begin{smallmatrix}a&b\\-c&-d\end{smallmatrix}\right)$ and $\left(\begin{smallmatrix}a&-b\\-c&d\end{smallmatrix}\right)$.

Now let $x\in[0,1]$ be given. Then for $Q\in\N$, let $\frac{a}{q}<\frac{a'}{q'}$ be consecutive in $\FF(Q)$ such that $\frac{a}{q}\leq x\leq\frac{a'}{q'}$, and note that the difference between $x$ and both $\frac{a}{q}$ and $\frac{a'}{q'}$ is at most $\frac{1}{Q}$. Thus any fraction between $\frac{a}{q}$ and $\frac{a'}{q'}$ is at most $\frac{1}{Q}$ away from $x$.

By our comments above, the coset $\left(\begin{smallmatrix}q&a\\q'&a'\end{smallmatrix}\right)H$ contains an element of the form $\left(\begin{smallmatrix}a_0&-b_0\\-c_0&d_0\end{smallmatrix}\right)$, with $0<a_0,b_0,c_0,d_0\leq k_H$. Thus $H$ contains
\[\left(\begin{array}{cc}a_0&-b_0\\-c_0&d_0\end{array}\right)^{-1}\left(\begin{array}{cc}q&a\\q'&a'\end{array}\right)=\left(\begin{array}{cc}d_0&b_0\\c_0&a_0\end{array}\right)\left(\begin{array}{cc}q&a\\q'&a'\end{array}\right)=\left(\begin{array}{cc}v&u\\v'&u'\end{array}\right),\]
where $\frac{u}{v}<\frac{u'}{v'}$ are consecutive Farey fractions of some order between $\frac{a}{q}$ and $\frac{a'}{q'}$ such that $v,v'\leq 2k_HQ$.
If $\left(\begin{smallmatrix}a_1&b_1\\-c_1&-d_1\end{smallmatrix}\right)$ is a representative of any coset $m_iH$ in $M$ with $0<a_1,b_1,c_1,d_1\leq k_H$, then $m_iH$ contains
\[\left(\begin{array}{cc}a_1&b_1\\-c_1&-d_1\end{array}\right)\left(\begin{array}{cc}v&u\\v'&u'\end{array}\right)=\left(\begin{array}{cc}a_1v+b_1v'&a_1u+b_1u'\\-c_1v-d_1v'&-c_1u-d_1u'\end{array}\right).\]
Thus $\FF_M(Q')$ contains $\frac{c_1u+d_1u'}{c_1v+d_1v'}$ for $Q'\geq\max\{c_1u+d_1u',c_1v+d_1v'\}$. Since $\max\{c_1u+d_1u',c_1v+d_1v'\}\leq4k_H^2Q$, $\FF_M(Q')$ contains $\frac{c_1u+d_1u'}{c_1v+d_1v'}$  for $Q'\geq 4k_H^2Q$. Thus the minimum distance between $x$ and an element in $\FF_M(4k_H^2Q)$ is at most $\frac{1}{Q}$. By letting $Q\geq8k_H^2$ and replacing $Q$ by $\left\lfloor\frac{Q}{4k_H^2}\right\rfloor$ in this argument, we see that the minimum distance between $x$ and an element in $\FF_M(Q)$ is at most $\frac{1}{(Q/4k_H^2)-1}\leq\frac{8k_H^2}{Q}$. This completes the proof with $Y=Q_0=8k_H^2$.
\end{proof}

This immediately implies the density of $(\FF_M(Q))$.

\begin{corollary}
The set $\bigcup_{Q\in\N}\FF_M(Q)$ is dense in $[0,1]$.
\end{corollary}

With Lemma \ref{L2} and the horocycle equidistribution results of Hejhal \cite{Hj}, and later Str\"ombergsson \cite{St}, it is conceivable that one can obtain results while allowing the subinterval $I$ to shrink with $Q$ and $n$ in Theorems \ref{T1} and \ref{T2}, respectively.

\section{A Poincar\'e section for $G/H$}\label{sec:4}

In the same way the properties of $\Omega'$ as a Poincar\'e section of the horocycle flow on $G/\Gamma$ were used in \cite{AC} to deduce many consequences for the gaps in $(\FF(Q))$, we now find a new space $\Omega_M$ that is a Poincar\'e section of the horocycle flow on $G/H$ which can be used to analyze the gaps in $\FF_M(Q)$. One step toward this goal is to lift the Poincar\'e section $\Omega'$ to $G/H$ via the natural projection $\pi:G/H\rar G/\Gamma$. In order to do this, we use the work of Fisher and Schmidt \cite{FS} on the behavior of a lifted Poincar\'e section for the geodesic flow from $F\bsl\PSL(2,\R)$, $F\subseteq\PSL(2,\R)$ being a Fuchsian group of finite covolume, to a finite cover $F'\bsl\PSL(2,\R)$ of $F\bsl\PSL(2,\R)$. In particular, we apply \cite[Lemma 2, Theorem 3]{FS} to the finite cover $\pi:G/H\rar G/\Gamma$, lifting the Poincar\'e section $\Omega'$ to $\Omega'':=\pi^{-1}(\Omega')$. In applying Theorem 3, we make the slight modifications of working in the left coset space instead of the right coset space, replacing the geodesic flow with the horocycle flow, and allowing the possibility that $-I\notin H$, so that $G/H$ is not necessarily of the form $\PSL(2,\R)/F'$. We summarize the results we need from this application in the following theorem:

\begin{theorem}
The set \[\Omega'':=\pi^{-1}(\Omega')=P\Gamma H/H=\{p_{a,b}gH:p_{a,b}\in P,g\in\Gamma\}\] is a Poincar\'e section for the action of $N$ on $G/H$ with first return time function $R\circ\pi:\Omega''\rar\R$ and first return map $r':\Omega''\rar\Omega''$ such that, for all $p_{a,b}\in P$ and $g\in\Gamma$,
\begin{align}
r'(p_{a,b}\gamma H)&=h_{R(\pi(p_{a,b}\gamma H))}p_{a,b}\gamma H\notag\\
&=\left(\begin{array}{cc} 1 & 0 \\ -\frac{1}{ab} & 1\end{array}\right)\left(\begin{array}{cc} a & b \\ 0 & a^{-1}\end{array}\right)\left(\begin{array}{cc} 0 & -1 \\ 1 & \lfloor\frac{1+a}{b}\rfloor\end{array}\right)\left(\begin{array}{cc} 0 & -1 \\ 1 & \lfloor\frac{1+a}{b}\rfloor\end{array}\right)^{-1}\gamma H\notag\\
&=\left(\begin{array}{cc} b & \lfloor\frac{1+a}{b}\rfloor b-a \\ 0 & b^{-1}\end{array}\right)\left(\begin{array}{cc} \lfloor\frac{1+a}{b}\rfloor & 1 \\ -1 & 0\end{array}\right)\gamma H.\label{e41}
\end{align}
Identify $\Omega''$ with $P\times\Gamma/H$ via the correspondence $p_{a,b}\gamma H\leftrightarrow(p_{a,b},\gamma H)$, and let $\mu_{\Omega''}$ be the measure on $\Omega''$ corresponding in this way to the product measure on $P\times\Gamma/H$ of $2\,da\,db$ with the counting measure on $\Gamma/H$. Then identify $G/H$ with $\{(x,s)\in\Omega''\times\R:0\leq s<(R\circ\pi)(x)\}$ via $(x,s)\leftrightarrow h_sx$. Then the Haar measure $\mu_{G/H}$ on $G/H$, normalized so that $\mu_{G/H}(G/H)=\frac{\pi^2}{3}[\Gamma:H]$, is given by $d\mu_{G/H}=d\mu_{\Omega''}\,ds$.
\end{theorem}

Our next step is to find a correspondence of $\FF(Q)$ with points in $\Omega''$ analogous to that of $\FF(Q)$ with the Farey points in $\Omega$. So suppose that $\frac{a}{q}<\frac{a'}{q'}<\frac{a''}{q''}$ are consecutive fractions in $\FF(Q)$. Then letting $K=\lfloor\frac{Q+q}{q'}\rfloor$ and noting that $a''=Ka'-a$ and $q''=Kq'-q$, by \eqref{e41} we have
\begin{align*}
r'\left(\left(\begin{array}{cc} \frac{q}{Q} & \frac{q'}{Q} \\ 0 & \frac{Q}{q}\end{array}\right)\left(\begin{array}{cc} q' & a' \\ -q & -a\end{array}\right)H\right)&=\left(\begin{array}{cc} \frac{q'}{Q} & \frac{Kq'-q}{Q} \\ 0 & \frac{Q}{q'}\end{array}\right)\left(\begin{array}{cc} Kq'-q & Ka'-a \\ -q' & -a'\end{array}\right)H\\
&=\left(\begin{array}{cc} \frac{q'}{Q} & \frac{q''}{Q} \\ 0 & \frac{Q}{q'}\end{array}\right)\left(\begin{array}{cc} q'' & a'' \\ -q' & -a'\end{array}\right)H.
\end{align*}
Therefore, if we associate each fraction $\frac{a}{q}\in\mathcal{F}(Q)$ to the element \[W_{H,Q}(\frac{a}{q})=\left(\begin{array}{cc} \frac{q}{Q} & \frac{q'}{Q} \\ 0 & \frac{Q}{q}\end{array}\right)\left(\begin{array}{cc} q' & a' \\ -q & -a\end{array}\right)H\]
of $\Omega''$, where $\frac{a'}{q'}$ is the element succeeding $\frac{a}{q}$ in $\FF(Q)$, then $r'(W_{H,Q}(\frac{a}{q}))=W_{H,Q}(\frac{a'}{q'})$. The map $W_{H,Q}$ gives the correspondence of $\FF(Q)$ with $\Omega''$ that we are seeking, and $W_{H,Q}$ sends a fraction in $\FF(Q)$ to \[\Omega_M:=PM=\{pm_iH:p\in P,m_iH\in M\}\]
if and only if the fraction is in $\FF_M(Q)$. The set $\Omega_M$ is the Poincar\'e section in $G/H$ we set out to find at the beginning of this section.

To see that $\Omega_M$ is in fact a Poincar\'e section for the horocycle flow on $G/H$, notice that the set
\[h_{[-1,0]}\Omega_M=\{h_spm_iH:s\in[-1,0],pm_iH\in\Omega_M\}\subseteq G/H\]
has positive $\mu_{G/H}$-measure. Now by the ergodicity of the horocycle flow \cite{He}, $\mu_{G/H}$-a.e.~$x\in G/H$ is sent to $h_{[-1,0]}\Omega_M$ by $\{h_s:s>0\}$. Clearly all of $h_{[-1,0]}\Omega_M$ is sent to $\Omega_M$ by $\{h_s:s\geq0\}$, and so a.e.~$x\in G/H$ is sent to $\Omega_M$ by $\{h_s:s>0\}$. The discreteness of $\{s\in\R:h_sx\in\Omega_M\}$ for a.e.~$x\in G/H$ follows from the fact that $\Omega_M\subseteq\Omega''$. This proves that $\Omega_M$ is a Poincar\'e section for the action of $N$ on $G/H$.

Let $R_M:\Omega_M\rar(0,\infty]$ be the first return time function $R_M(x)=\min\{s>0:h_sx\in\Omega_M\}$ and $r_M$ be the first return map on $\Omega_M$ defined by $r_M(x)=h_{R_M(x)}(x)$. Here we note that if $\mu_{\Omega_M}=\frac{1}{\#M}\mu_{\Omega''}|_{\Omega_M}$, then $d\mu_{G/H}=(\#M)\,d\mu_{\Omega_M}\,ds$, where we identify $G/H$ with $\{(x,s)\in\Omega_M\times\R:0\leq s<R_M(x)\}$ by $h_sx\leftrightarrow(x,s)$.

Now that we have identified a Poincar\'e section in $G/H$ and a correspondence $W_{H,Q}$ of $\FF_M(Q)$ to a subset of $\Omega_M$ analogous to the set of Farey points in $\Omega$, we wish to see if information about the gaps in $\FF_M(Q)$ can be deduced from $\Omega_M$. So let $\gamma_i\in\FF_M(Q)$ such that there are $i'>i$ with $\gamma_{i'}\in\FF_M(Q)$. Since $\Omega_M\subseteq\Omega''$, for each $x\in\Omega_M$ in which $r_M(x)$ is defined, $r_M(x)=r'^j(x)$ for some $j\in\N$. So \[r_M(W_{H,Q}(\gamma_i))=r'^j(W_{H,Q}(\gamma_i))=W_{H,Q}(\gamma_{i+j}),\]
where $j\in\N$ is the least element such that $W_{H,Q}(\gamma_{i+j})\in\Omega_M$, i.e., $\gamma_{i+j}\in\FF_M(Q)$. We then have
\begin{align}
R_M(W_{H,Q}(\gamma_i))&=\sum_{i'=0}^{j-1}(R\circ\pi)(r'^{i'}(W_{H,Q}(\gamma_i)))=\sum_{i'=0}^{j-1}\frac{Q^2}{q_{i+i'}q_{i+i'+1}}\notag\\
&=Q^2(\gamma_{i+j}-\gamma_i).\label{e42}
\end{align}
So just as the return time function $R$ on $\Omega'$ contained information about the gaps in $\FF(Q)$, $R_M$ contains information about the gaps in $\FF_M(Q)$.

Let $N_{I,M}(Q)=\#\FF_{I,M}(Q)-1$ and
\[\FF_{I,M}(Q)=\{\beta_0<\beta_1<\cdots<\beta_{N_{I,M}(Q)}\}.\]
For notational convenience, we suppress the dependence of the $\beta_i$ on $Q$. Then define the measure $\rho_{Q,I,M}$ on $\Omega_M$ by \[\rho_{Q,I,M}=\frac{1}{N_{I,M}(Q)}\sum_{i=0}^{N_{I,M}(Q)-1}\delta_{W_{H,Q}(\beta_i)}.\]
By \eqref{e42}, we have
\begin{align}
\frac{\#\{0\leq i\leq N_{I,M}(Q)-1:Q^2(\beta_{i+1}-\beta_i)\in[0,c]\}}{N_{I,M}(Q)}=\rho_{Q,I,M}(R_M^{-1}[0,c])\label{e43}
\end{align}
for all $c\geq0$. To show that the limit of the left side, and hence the right side, exists, we prove in Section \ref{sec:5} that the boundary $\partial(R_M^{-1}[0,c])$ of the set $R_M^{-1}[0,c]$ has $\mu_{\Omega_M}$-measure $0$, and then prove in Section \ref{sec:6} that $(\rho_{Q,I,M})$ converges weakly to $\mu_{\Omega_M}$ as $Q\rar\infty$. These results imply that
\[\lim_{Q\rar\infty}\rho_{Q,I,M}(R_M^{-1}[0,c])=\mu_{\Omega_M}(R_M^{-1}[0,c])\]
by the Portmanteau theorem (see \cite{B}). Note that the right side of \eqref{e43} is not the relevant limit to prove the existence of the limiting gap measure for $(\FF_{I,M}(Q))$. However, we have $\beta_{N_{I,M}(Q)}-\beta_0\rar|I|$ as $Q\rar\infty$ by the density of $\bigcup_{Q\in\N}\FF_M(Q)$ in $[0,1]$, and we show that $N_{I,M}(Q)\sim\frac{|I|(\#M)Q^2}{\mu_{G/H}(G/H)}=\frac{3|I|(\#M)Q^2}{\pi^2[\Gamma:H]}$ in the course of our work in Section \ref{sec:6}. It then follows that the limiting gap measure $\nu_{I,M}$ exists and satisfies
\begin{equation}
\nu_{I,M}([0,c])=\mu_{\Omega_M}\left(R_M^{-1}\left[0,\frac{\pi^2[\Gamma:H]}{3(\#M)}c\right]\right)\label{e44}
\end{equation} for all $c\geq0$. Another corollary is that $\lim_{Q\rar\infty}\frac{N_{I,M}(Q)}{N_{[0,1],M}(Q)}=|I|$ for every subinterval $I\subseteq[0,1]$, implying that $(\FF_M(Q))$ becomes equidistributed in $[0,1]$ as $Q\rar\infty$.

\section{The return time function $R_M$}\label{sec:5}

In this section, we prove important properties of the first return time function $R_M$. Specifically, we show that $\mu_{\Omega_M}(\partial(R_M^{-1}[0,c]))=0$ for every $c\geq0$, and the function $F_M:[0,\infty)\rar[0,1]$ defined by
\[F_M(c)=\mu_{\Omega_M}(R_M^{-1}[0,c])\] has a continuous, piecewise real-analytic derivative. These results, together with our work in Section \ref{sec:6}, proves that the limiting gap measure $\nu_{I,M}$ exists and has a continuous and piecewise real-analytic density. To do this, we show that $R_M$ is a piecewise rational function, viewing each component $Pm_iH/H$ of $\Omega_M$ as a copy of the Farey triangle $\Omega$ by the correspondence $p_{a,b}m_iH\leftrightarrow(a,b)$, and that the region in a given component $Pm_iH/H$ over which $R_M$ is equal to a certain rational function is a polygon. This allows us to say that $R_M^{-1}[0,c]$ is the union of regions, each being obtained by intersecting a polygon with a region bounded below by a hyperbola which depends in a smooth way on $c$. In particular, we show that $R_M^{-1}[0,c]$ is a finite union of these regions, which grants us the properties of $F_M$ we seek.

\subsection{$R_M$ is piecewise rational}\label{sec:5.1}

First let $p_{a,b}\in P$ and $m_iH\in M$ so that $p_{a,b}m_iH\in\Omega_M$, and suppose $s>0$. We have $h_sp_{a,b}m_iH\in\Omega_M$ if and only if there exist $p_{c,d}\in P$ and $m_jH\in M$ such that $h_sp_{a,b}m_iH=p_{c,d}m_jH$. This means there exists $h\in H$ such that $h_sp_{a,b}m_ihm_j^{-1}=p_{c,d}$. Letting $m_ihm_j^{-1}=\left(\begin{smallmatrix} c_1 & c_2 \\ c_3 & c_4\end{smallmatrix}\right)$, this equality is
\[\left(\begin{array}{cc} ac_1+bc_3 & ac_2+bc_4 \\ a^{-1}c_3-s(ac_1+bc_3) & a^{-1}c_4-s(ac_2+bc_4)\end{array}\right)=\left(\begin{array}{cc} c & d \\ 0 & c^{-1}\end{array}\right).\]
Thus $h_sp_{a,b}m_iH\in\Omega_M$ if and only if there exists $\left(\begin{smallmatrix}c_1&c_2\\ c_3&c_4\end{smallmatrix}\right)\in\bigcup_{j=1}^km_iHm_j^{-1}$ such that $(ac_1+bc_3,ac_2+bc_4)\in\Omega$ and $s=\frac{c_3}{a(ac_1+bc_3)}=\frac{1}{a(b+\frac{c_1}{c_3}a)}$. Note that the latter conditions and $s>0$ imply that $c_3>0$, and hence $c_1\leq0$ since $ac_1+bc_3\leq1$ and $a+b>1$. In particular, if $R_M(p_{a,b}m_iH)<\infty$, then $R_M(p_{a,b}m_iH)=\frac{1}{a(b+\frac{c_1}{c_3}a)}$, where $\frac{c_1}{c_3}$ is the greatest fraction such that $c_1\leq0$, $c_3>0$, and there exists $\left(\begin{smallmatrix}c_1&c_2\\ c_3&c_4\end{smallmatrix}\right)\in\bigcup_{j=1}^km_iHm_j^{-1}$ with $(ac_1+bc_3,ac_2+bc_4)\in\Omega$. We have thus proven the following result:

\begin{proposition}
The function $R_M$ is a piecewise rational function. Specifically,
\[R_M=\min_{C\in\bigcup_{i,j=1}^km_iHm_j^{-1}}f_C,\]
where for each $C=\left(\begin{smallmatrix}c_1&c_2\\ c_3&c_4\end{smallmatrix}\right)\in\bigcup_{i,j=1}^km_iHm_j^{-1}$, $f_C:\Omega_M\rar[0,\infty]$ is defined by
\[f_C(p_{a,b}m_iH)=\begin{cases}\frac{1}{a(b+\frac{c_1}{c_3}a)} & \mbox{if }C\in\bigcup_{j=1}^km_iHm_j^{-1},c_3>0,(\,a\,\,b\,)\left(\begin{smallmatrix}c_1&c_2\\c_3&c_4\end{smallmatrix}\right)\in\Omega\\
\infty & \mbox{otherwise.}
\end{cases}\]
\end{proposition}

Next, for a given fraction $\frac{c_1}{c_3}$ with $c_1\leq0$ and $c_3>0$, we wish to better understand the region
\[R_{-c_1/c_3}:=\left\{p_{a,b}m_iH\in\Omega_M:R_M(p_{a,b}m_iH)=\frac{1}{a(b+\frac{c_1}{c_3}a)}\right\}.\]
In particular, we want to prove the following technical result, which is a great aid in showing that $\mu_{\Omega_M}(\partial(R_M^{-1}[0,c]))=0$ for $c\geq0$ and that $F_M'$ is continuous and piecewise analytic.

\begin{proposition}
For each $i\in\{1,\ldots,k\}$, $R_{-c_1/c_3}\cap Pm_iH/H$ is either empty or a polygon, viewing $Pm_iH/H$ as the Farey triangle $\Omega$.
\end{proposition}

\begin{proof}
We first examine the region
\[R_C:=\left\{p_{a,b}m_iH\in\Omega_M:(ac_1+bc_3,ac_2+bc_4)\in\Omega,C\in\bigcup_{j=1}^km_iHm_j^{-1}\right\}\]
for a given $C=\left(\begin{smallmatrix}c_1&c_2\\ c_3&c_4\end{smallmatrix}\right)\in\bigcup_{i,j=1}^km_iHm_j^{-1}$. Note that $R_{-c_1/c_3}$ is a subset of the union of all $R_{C'}$ such that $C'\in\bigcup_{i,j=1}^km_iHm_j^{-1}$ is a matrix having $\left(\begin{smallmatrix}c_1\\c_3\end{smallmatrix}\right)$ as its first column. If $c_4=0$, then $C=\left(\begin{smallmatrix}c_1&-1\\1&0\end{smallmatrix}\right)$ and for a given index $i$, $R_C\cap Pm_iH/H$ is a subset of $\{p_{a,b}m_iH:(ac_1+c_3,-a)\in\Omega\}=\emptyset$. Thus $R_C$ is empty in this case. Next, suppose $c_4\leq-1$. Then $R_C\cap Pm_iH/H$ consists of the elements $p_{a,b}m_iH$ which must satisfy \[\begin{array}{ccc}-\frac{c_1}{c_3}a<b\leq\frac{1-c_1a}{c_3} & \mbox{and} & \frac{1-c_2a}{c_4}\leq b<-\frac{c_2}{c_4}a.\end{array}\] However, we have $c_1c_4-c_2c_3=1$, which implies that $\frac{c_1}{c_3}-\frac{c_2}{c_4}=\frac{1}{c_3c_4}<0$, and hence $-\frac{c_2}{c_4}<-\frac{c_1}{c_3}$. Thus the above conditions on $p_{a,b}m_iH$ cannot be satisfied, and therefore $R_C=\emptyset$. So $R_C$ is nonempty only when $c_1\leq0$ and $c_3,c_4\geq1$, which then implies that $c_2=\frac{c_1c_4-1}{c_3}<0$.

As a result of our work above, we reset notation so that $C=\left(\begin{smallmatrix}-c_1&-c_2\\ c_3&c_4\end{smallmatrix}\right)$ and we assume that $c_1\geq0$ and $c_2,c_3,c_4\geq1$. Now a point $(a,b)\in\Omega$ satisfies $(-ac_1+bc_3,-ac_2+bc_4)\in\Omega$ if and only if
\[\frac{c_1}{c_3}a<b\leq\frac{1+c_1a}{c_3},\quad \frac{c_2}{c_4}a<b\leq\frac{1+c_2a}{c_4},\quad\mbox{and}\quad b>\frac{1+(c_1+c_2)a}{c_3+c_4}.\]
Since $\frac{c_1}{c_3}a<\frac{c_2}{c_4}a\leq\frac{1+(c_1+c_2)a}{c_3+c_4}$ for $a\in(0,1]$, the above conditions reduce to $\frac{1+(c_1+c_2)a}{c_3+c_4}<b\leq\min\{\frac{1+c_1a}{c_3},\frac{1+c_2a}{c_4}\}$. We therefore have
\[R_C=\left\{p_{a,b}m_iH\in\Omega_M:\begin{array}{c}\frac{1+(c_1+c_2)a}{c_3+c_4}<b\leq\min\{\frac{1+c_1a}{c_3},\frac{1+c_2a}{c_4}\}\\ C\in\bigcup_{j=1}^km_iHm_j^{-1}\end{array}\right\}.\]

Since $M$ is closed under left multiplication by $\left(\begin{smallmatrix}1&-1\\0&1\end{smallmatrix}\right)$, $\bigcup_{j=1}^km_iHm_j^{-1}$ is closed under right multiplication by $\left(\begin{smallmatrix}1&1\\0&1\end{smallmatrix}\right)$ for each $i\in\{1,\ldots,k\}$. So if $C\in\bigcup_{j=1}^km_iHm_j^{-1}$, then for every $n\in\N$, \[C\left(\begin{array}{cc}1&n\\0&1\end{array}\right)=\left(\begin{array}{cc}-c_1&-nc_1-c_2\\c_3&nc_3+c_4\end{array}\right)\in\bigcup_{j=1}^km_iHm_j^{-1}.\]
We have
\[R_{C\left(\begin{smallmatrix}1&n\\0&1\end{smallmatrix}\right)}=\left\{p_{a,b}m_iH\in\Omega_M:\begin{array}{c}\frac{1+((n+1)c_1+c_2)a}{(n+1)c_3+c_4}<b\leq\frac{1+(nc_1+c_2)a}{nc_3+c_4}\\ C\in\bigcup_{j=1}^km_iHm_j^{-1}\end{array}\right\},\]
noting that $\frac{1+(nc_1+c_2)a}{nc_3+c_4}<\frac{1+c_3a}{c_4}$. Thus for each $i$ such that $C\in\bigcup_{j=1}^km_iHm_j^{-1}$, the regions $\{R_{C\left(\begin{smallmatrix}1&n\\0&1\end{smallmatrix}\right)}\cap Pm_iH/H:n\in\N\}$ paste together to form
\[\left\{p_{a,b}m_iH\in Pm_iH/H:\frac{c_1}{c_3}a<b\leq\min\Big\{\frac{1+c_1a}{c_3},\frac{1+c_2a}{c_4}\Big\}\right\}\]
since the sequence $(\frac{1+(nc_1+c_2)a}{nc_3+c_4})$ decreases to $\frac{c_1}{c_3}a$ for each $a\in(0,1]$. Therefore, if we assume that $\frac{c_{2,i}}{c_{4,i}}$ is the largest fraction such that $\left(\begin{smallmatrix}-c_1&-c_{2,i}\\c_3&c_{4,i}\end{smallmatrix}\right)\in\bigcup_{j=1}^km_iHm_j^{-1}$, then
\[R_{c_1/c_3}\cap Pm_iH/H\subseteq\left\{p_{a,b}m_iH:\frac{c_1}{c_3}a<b\leq\min\Big\{\frac{1+c_1a}{c_3},\frac{1+c_{2,i}a}{c_{4,i}}\Big\}\right\}.\]
We use $R_{c_1/c_3}^{(i)}$ to denote the set on the right for each $c_1\geq0$, $c_3\geq1$, and $i\in\{1,\ldots,k\}$ such that $\left(\begin{smallmatrix}-c_1\\c_3\end{smallmatrix}\right)$ is the first column of a matrix in $\bigcup_{j=1}^km_iHm_j^{-1}$, and $\frac{c_{2,i}}{c_{4,i}}$ ($c_{2,i},c_{4,i}\geq1$) is the largest fraction with $\left(\begin{smallmatrix}-c_1&-c_{2,i}\\c_3&c_{4,i}\end{smallmatrix}\right)\in\bigcup_{j=1}^km_iHm_j^{-1}$. If $c_1$, $c_3$, and $i$ do not satisfy these conditions, we let $R_{c_1/c_3}^{(i)}=\emptyset$. Then for all $c_1$ and $c_3$, we have \[R_{c_1/c_3}=\bigcup_{i=1}^k\left(R_{c_1/c_3}^{(i)}\setminus\bigcup_{s\in\Q,0\leq s<c_1/c_3}R_s^{(i)}\right).\]

Assume that $c_1$, $c_3$, and $i$ are such that $R_{c_1/c_3}^{(i)}\neq\emptyset$. In order to show that $R_{c_1/c_3}\cap Pm_iH/H$ is either empty or a polygon, it is sufficient to prove that $R_{c_1/c_3}\cap Pm_iH/H$ can be written in the form $R_{c_1/c_3}^{(i)}\bsl(\bigcup_{\ell=1}^nR_{s_\ell}^{(i)})$ for some $s_\ell\in\Q$. Note that if $\frac{c_1}{c_3}=0$, then $R_{c_1/c_3}\cap Pm_iH/H=R_0^{(i)}$, which is a triangle if nonempty. So assume that $\frac{c_1}{c_3}>0$. We now consider two cases.

\noindent \textit{Case 1.} There exists $s\in\Q$ with $s<\frac{c_1}{c_3}$ such that $R_s^{(i)}$ contains the lower-right border $\{p_{a,b}m_iH:b=\frac{c_1}{c_3}a\}$ of $R_{c_1/c_3}^{(i)}$ in its interior. Then clearly there exists $N\in\N$ such that \[R_{c_1/c_3}^{(i)}\bsl R_s^{(i)}\subseteq\left\{p_{a,b}m_iH:\frac{1}{N}+\frac{c_1}{c_3}a<b\leq\frac{1+c_1a}{c_3}\right\}.\]
If there is another $\frac{c_1'}{c_3'}\in\Q$ with $\frac{c_1'}{c_3'}<\frac{c_1}{c_3}$ such that $R_{c_1'/c_3'}^{(i)}$ intersects $R_{c_1/c_3}^{(i)}\bsl R_s^{(i)}$, then there exists $(a,b)\in\Omega$ such that $\frac{1}{N}+\frac{c_1}{c_3}a<b\leq\frac{1+c_1'a}{c_3'}$. This implies that $(\frac{c_1}{c_3}-\frac{c_1'}{c_3'})a<\frac{1}{c_3'}-\frac{1}{N}$, which can only hold if $c_3'<N$. Thus there are finitely many $s'\in\Q$ such that $R_{s'}^{(i)}$ intersects $R_{c_1/c_3}^{(i)}\bsl R_s^{(i)}$. So \[R_{c_1/c_3}\cap Pm_iH/H=R_{c_1/c_3}^{(i)}\bsl\left(R_s^{(i)}\cup\bigcup_{\ell=1}^nR_{s_\ell}^{(i)}\right)\] for some $s_\ell\in\Q$, completing the proof that $R_{c_1/c_3}\cap Pm_iH/H$ is a polygon in this case.

\noindent \textit{Case 2.} There exists no $s\in\Q$ with $s<\frac{c_1}{c_3}$ such that $R_s^{(i)}$ contains the lower-right border of $R_{c_1/c_3}^{(i)}$ in its interior. Let $n_1\geq0$ and $n_2\geq1$ be integers such that $\{p_{a,b}m_iH:b=\frac{1+n_1a}{n_2}\}$ is an upper border for $R_{s}^{(i)}$ for some $s\in\Q$ with $s<\frac{c_1}{c_3}$. If $\frac{n_1}{n_2}=\frac{c_1}{c_3}$, then the line $\{(a,b)\in\R^2:b=\frac{1+n_1a}{n_2}\}$ is above and parallel to $\{(a,b)\in\R^2:b=\frac{c_1}{c_3}a\}$, in which case the lower-right border of $R_{c_1/c_3}^{(i)}$ is contained in the interior of $R_s^{(i)}$, a contradiction. So $\frac{n_1}{n_2}\neq\frac{c_1}{c_3}$, and $\{(a,b)\in\R^2:b=\frac{1+n_1a}{n_2}\}$ intersects $\{(a,b)\in\R^2:b=\frac{c_1}{c_3}a\}$ at the point $(\frac{c_3}{n_2c_1-n_1c_3},\frac{c_1}{n_2c_1-n_1c_3})$. If $\frac{n_1}{n_2}>\frac{c_1}{c_3}$, then since $\frac{c_3}{n_2c_1-n_1c_3}<0$ and the slope of $b=\frac{1+n_1a}{n_2}$ is greater than that of $b=\frac{c_1}{c_3}a$, the lower-right border of $R_{c_1/c_3}^{(i)}$ is again contained in the interior of $R_s^{(i)}$, another contradiction.

\begin{figure}
\centering
\begin{minipage}{0.5\textwidth}
  \centering
\begin{overpic}[width=0.98\linewidth]{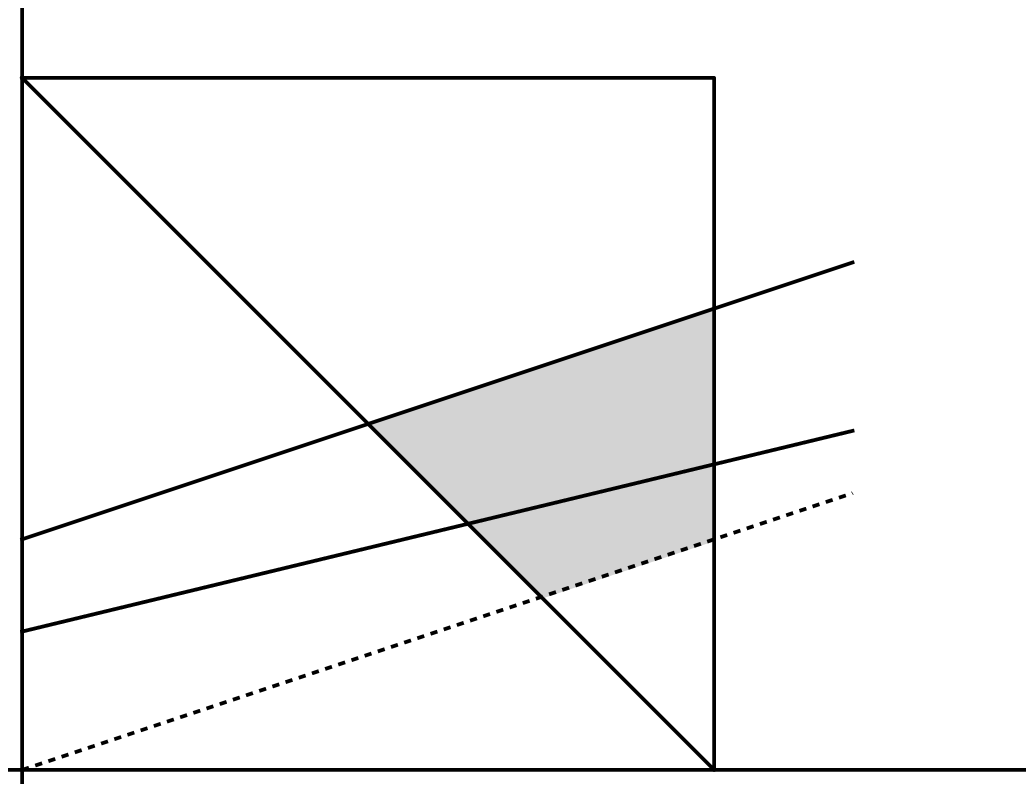}
\put (83,54){{\footnotesize $b=\frac{1+c_1a}{c_3}$}}
\put (83,39){{\footnotesize $b=\frac{1+n_1a}{n_2}$}}
\put (83,30){{\footnotesize $b=\frac{c_1}{c_3}a$}}
\end{overpic}
  \caption{$R_{c_1/c_3}$ in Case 1}

\end{minipage}%
\begin{minipage}{0.5\textwidth}
  \centering
\begin{overpic}[width=0.98\linewidth]{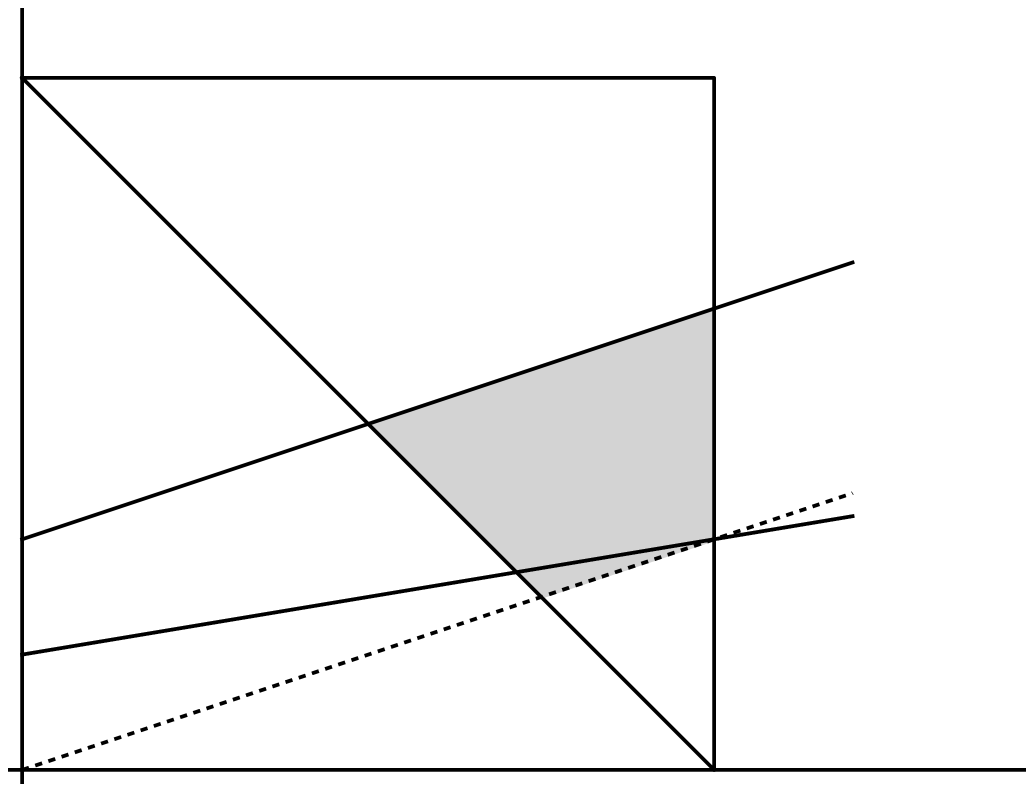}
\put (83,54){{\footnotesize $b=\frac{1+c_1a}{c_3}$}}
\put (83,26){{\footnotesize $b=\frac{1+n_1a}{n_2}$}}
\put (83,34){{\footnotesize $b=\frac{c_1}{c_3}a$}}
\end{overpic}
  \caption{$R_{c_1/c_3}$ in Case 2}
\end{minipage}
\end{figure}

So $\frac{n_1}{n_2}<\frac{c_1}{c_3}$, and furthermore, the intersection point $(\frac{c_3}{n_2c_1-n_1c_3},\frac{c_1}{n_2c_1-n_1c_3})$ cannot be above or to the right of $\Omega$. If the intersection point is in the interior of $\Omega$ or is on or below its border $b=1-a$, then $R_{c_1/c_3}^{(i)}\bsl R_s^{(i)}$ contains a set of the form \[\{p_{a,b}m_iH:b-\frac{c_1}{c_3}a\in(0,\epsilon),a\in(t_1,t_2)\}\] where $\epsilon,t_1,t_2\in(0,1)$ and $t_1<t_2$. Since $\frac{1}{a(b-\frac{c_1}{c_3}a)}$ is not $\mu_{\Omega_M}$-integrable over the above set for any $\epsilon,t_1,t_2\in(0,1)$ with $t_1<t_2$ and $R_M$ is $\mu_{\Omega_M}$-integrable, the lower-right border of $R_{c_1/c_3}^{(i)}$ is contained in $\bigcup_{s\in\Q,0\leq s<\frac{c_1}{c_3}}R_s^{(i)}$. Since the set $\{(\frac{c_3}{n},\frac{c_1}{n}):n\in\N\}\cap\Omega$ of possible intersection points in $\Omega$ of $b=\frac{c_1}{c_3}a$ with a line $b=\frac{1+n_1a}{n_2}$ corresponding to the upper-left border of a set $R_s^{(i)}$ is finite, there must be some $s\in\Q$ with $s<\frac{c_1}{c_3}$ such that $R_s^{(i)}$ contains the lower-right border of $R_{c_1/c_3}^{(i)}$, and in this case, the line determining the upper-left border of $R_s^{(i)}$, say $b=\frac{1+n_1a}{n_2}$, intersects $b=\frac{c_1}{c_3}a$ at the border of $\Omega$ at the line $a=1$ or $b=1$.

Now let $s'\in\Q$ with $s'<\frac{c_1}{c_3}$ such that the upper-left border of $R_{s'}^{(i)}$ is determined by $b=\frac{1+n_1'a}{n_2'}$. Since $R_{s'}^{(i)}$ does not contain the lower-right border of $R_{c_1/c_3}^{(i)}$, $\frac{n_1'}{n_2'}<\frac{c_1}{c_3}$ and $b=\frac{1+n_1'a}{n_2'}$ intersects $b=\frac{c_1}{c_3}a$ at or to the left of the intersection point of $b=\frac{1+n_1a}{n_2}$ with $b=\frac{c_1}{c_3}a$. If $\frac{n_1'}{n_2'}>\frac{n_1}{n_2}$, then it is clear that the line $b=\frac{1+n_1'a}{n_2'}$ passes under the set $\{(a,b)\in\Omega:b>\max\{\frac{1+n_1a}{n_2},\frac{c_1}{c_3}a\}\}$, and thus $R_{s'}^{(i)}$ does not intersect $R_{c_1/c_3}^{(i)}\bsl R_s^{(i)}$. If $\frac{n_1'}{n_2'}<\frac{n_1}{n_2}$ and $R_{s'}^{(i)}$ does intersect $R_{c_1/c_3}^{(i)}\bsl R_s^{(i)}$, then there exists $(a,b)\in\Omega$ such that $\frac{1+n_1a}{n_2}<b\leq\frac{1+n_1'a}{n_2'}$, implying that $(\frac{n_1}{n_2}-\frac{n_1'}{n_2'})a<\frac{1}{n_2'}-\frac{1}{n_2}$. This inequality holds only if $n_2'<n_2$. This shows that there are finitely many $s'\in\Q$ such that $R_{s'}^{(i)}$ intersects $R_{c_1/c_3}^{(i)}\bsl R_s^{(i)}$, and thus completes the proof that $R_{c_1/c_3}\cap Pm_iH/H$ is a polygon.

\end{proof}

So by our work above, we have proven that $R_M$ is a piecewise rational function on $\Omega_M$, and for a given $\frac{c_1}{c_3}\in\Q$, the region $R_{c_1/c_3}$ over which $R_M(p_{a,b}m_iH)=\frac{1}{a(b-\frac{c_1}{c_3}a)}$ is either empty or a union of polygons, one polygon being in each component $Pm_iH/H$ such that $R_{c_1/c_3}\cap Pm_iH/H\neq\emptyset$. Note also that if $C=\left(\begin{smallmatrix}-c_1&-c_2\\c_3&c_4\end{smallmatrix}\right)\in m_iHm_j^{-1}$, then the restriction of the return map $r_M$ to the polygon $R_{c_1/c_3}\cap R_C\cap Pm_iH/H$ is given by
\[r_M(p_{a,b}m_iH)=p_{-c_1a+c_3b,-c_2a+c_4b}m_jH.\]
Therefore, each component $Pm_iH/H$ of $\Omega_M$ can be divided into a countable number of polygons $P'$ such that $r_M$ maps each $P'$ linearly from $Pm_iH/H$ to $Pm_jH/H$, for some $j$ depending on $P'$.

\subsection{The boundary of $R_M^{-1}[0,c]$ has measure $0$}\label{sec:5.2}

Next, we prove that for a given $c>0$, the boundary of $R_M^{-1}[0,c]$ has $\mu_{\Omega_M}$-measure $0$. First notice that $R_M$ is continuous $\mu_{\Omega_M}$-a.e. Because the set $\ol{R_M^{-1}[0,c]}\bsl R_M^{-1}[0,c]$ contains only points of discontinuity of $R_M$, it has $\mu_{\Omega_M}$-measure $0$. Now for a given $s\in\Q$, define $f_s:\Omega_M\rar[0,\infty]$ so that
\[f_s(p_{a,b}m_iH)=\begin{cases}\frac{1}{a(b-sa)} & \mbox{for all $p_{a,b}m_iH\in R_s$}\\
\infty & \mbox{otherwise.}\end{cases}\]
We then have
\begin{align*}
R_M^{-1}[0,c]\bsl (R_M^{-1}[0,c])^o&=\bigcup_{s\in\Q}f_s^{-1}[0,c]\bsl\left(\bigcup_{s\in\Q}f_s^{-1}[0,c]\right)^o\\
&\subseteq\bigcup_{s\in\Q}f_s^{-1}[0,c]\bsl\left(\bigcup_{s\in\Q}(f_s^{-1}[0,c])^o\right)\\
&\subseteq\bigcup_{s\in\Q}((f_s^{-1}[0,c])\bsl(f_s^{-1}[0,c])^o).
\end{align*}
Each set $f_s^{-1}[0,c]$ is either empty or a finite union of sets of the form
\[\left\{p_{a,b}m_iH\in Pm_iH:\frac{c_1}{c_3}a<b\leq\min\Big\{\frac{1+c_1a}{c_3},\frac{1+c_{2,i}a}{c_{4,i}}\Big\},b\geq\frac{c_1}{c_3}a+\frac{1}{ca}\right\}.\]
So $(f_s^{-1}[0,c])\bsl(f_s^{-1}[0,c])^o$ is essentially a set of finitely many line and curve segments for each $s\in\Q$. Hence it is clear that $R_M^{-1}[0,c]\bsl(R_M^{-1}[0,c])^o$ is of $\mu_{\Omega_M}$-measure $0$. As a consequence,
\[\partial(R_M^{-1}[0,c])=\ol{R_M^{-1}[0,c]}\bsl(R_M^{-1}[0,c])^o=(\ol{R_M^{-1}[0,c]}\bsl R_M^{-1}[0,c])\cup (R_M^{-1}[0,c]\bsl(R_M^{-1}[0,c])^o)\]
has $\mu_{\Omega_M}$-measure $0$. This, along with Section \ref{sec:6}, proves the existence of the limiting gap measure $\nu_{I,M}$.

\subsection{$F_M'$ is continuous and piecewise real-analytic}\label{sec:5.3}

Next, we want to prove that the function $F_M$ has a continuous, piecewise real-analytic derivative. Note first that since $R_M^{-1}[0,c]$ is the disjoint union of the sets $f_s^{-1}[0,c]$, we have
\[F_M(c)=\sum_{s\in\Q}\mu_{\Omega_M}(f_s^{-1}[0,c]).\]
So it suffices to show that $c\mapsto\mu_{\Omega_M}(f_s^{-1}[0,c])$ has a continuous, piecewise real-analytic derivative for every $s\in\Q$, and that for a given $c>0$, there are at most finitely many $s\in\Q$ for which $f_s^{-1}[0,c]$ is nonempty.

The first claim is easy to see. Indeed, by triangulating the polygons that make up the region $R_s$, we can write $c\mapsto\mu_{\Omega_M}(f_s^{-1}[0,c])$ as a finite sum of functions $g_T:[0,\infty)\rar\R$ defined by
\[g_T(c)=\frac{1}{\#M}m\left(\Big\{(a,b)\in T:b\geq sa+\frac{1}{ca}\Big\}\right),\]
where $T\subseteq\Omega$ is a triangle and as in Section \ref{sec:2}, $dm=2\,da\,db$. It is then straightforward to show that each function $g_T$ has a continuous, piecewise real-analytic derivative, implying that $c\mapsto\mu_{\Omega_M}(f_s^{-1}[0,c])$ has the same property.

To prove the latter assertion, let $c>0$ be given and suppose that there exists $\frac{c_1}{c_3}\in\Q$ such that $f_{c_1/c_3}^{-1}[0,c]$ is nonempty (assume $c_1\geq0$ and $c_3\geq1$). Then there is some index $i$ such that $f_{c_1/c_3}^{-1}[0,c]\cap Pm_iH/H\neq\emptyset$. We have
\[f_{c_1/c_3}^{-1}[0,c]\cap Pm_iH/H\subseteq\left\{p_{a,b}m_iH:\frac{c_1}{c_3}a<b\leq\frac{1+c_1a}{c_3},b\geq\frac{c_1}{c_3}a+\frac{1}{ca}\right\},\]
and hence the latter set is nonempty. So there exists $p_{a,b}m_iH\in Pm_iH/H$ such that $\frac{c_1}{c_3}a+\frac{1}{ca}\leq b\leq\frac{1+c_1a}{c_3}$, which implies that $c\geq\frac{c_3}{a}$. We have
\[\sup\left\{a\in\R:(a,b)\in\Omega,\frac{c_1}{c_3}a<b\leq\frac{1+c_1a}{c_3}\right\}=\begin{cases}1&\mbox{if $\frac{c_1}{c_3}\leq1$}\\\frac{c_3}{c_1}&\mbox{if $\frac{c_1}{c_3}>1$}\end{cases},\]
and therefore $c\geq c_3$ if $\frac{c_1}{c_3}\leq1$ and $c\geq c_1$ if $\frac{c_1}{c_3}>1$; i.e., $c\geq\max\{c_1,c_3\}$. There are clearly finitely many positive fractions $\frac{c_1}{c_3}$ satisfying this condition, and thus satisfying $f_{c_1/c_3}^{-1}([0,c])\neq\emptyset$. This completes the proof that $F_M'$ is continuous and piecewise real-analytic.

One more property of $R_M$ we wish to mention is that for any $\epsilon>0$, there exist bounded continuous functions $g_{1,\epsilon},g_{2,\epsilon},g_{3,\epsilon}:\Omega_M\rar[0,\infty)$ such that $g_{1,\epsilon}\leq R_M$, $g_{2,\epsilon}\leq\frac{1}{R_M}\leq g_{3,\epsilon}$, and
\[\int_{\Omega_M}(R_M-g_{1,\epsilon})\,d\mu_{\Omega''},\int_{\Omega_M}\left(\frac{1}{R_M}-g_{2,\epsilon}\right)\,d\mu_{\Omega''},\int_{\Omega_M}\left(g_{3,\epsilon}-\frac{1}{R_M}\right)\,d\mu_{\Omega''}<\epsilon.\]
This follows easily from the properties of $R_M$ proven in Section \ref{sec:5.1}, and the fact that $\frac{1}{R_M}\leq\frac{1}{r}\leq1$.

\subsection{The $h$-spacings and numerators of differences in $(\FF_{I,M}(Q))$}\label{sec:5.4}

For $h\in\N$ and $A=\{x_0\leq x_1\leq\cdots\leq x_N\}\subseteq[0,1]$, let $\mathbf{v}_{A,i,h}=(x_{i+j}-x_{i+j-1})_{j=1}^h\in\R^h$ for $i\in\{1,\ldots,N-h\}$. We define the \textit{$h$-spacing distribution measure} of $A$ to be the measure $\nu_{A,h}$ on $[0,\infty)^h$ such that
\[\nu_{A,h}\left(\prod_{j=1}^h[0,c_j]\right)=\frac{1}{N}\#\left\{x_i\in A:N\mathbf{v}_{A,i,h}\in\prod_{j=1}^h[0,c_j(x_N-x_0)]\right\}.\]
For an increasing sequence $(A_n)$ of subsets of $[0,1]$, we call the weak limit of $(\nu_{A_n,h})$, if it exists, the \textit{limiting $h$-spacing measure} of $(A_n)$.

Upon $(R_M\circ r_M^{j-1})(W_{H,Q}(\beta_i))=Q^2(\beta_{i+j}-\beta_{i+j-1})$, $i\in\{0,\ldots,N_{I,M}(Q)-h\}$, we have
\[\frac{\#\{\beta_i\in\FF_{I,M}(Q):Q^2\mathbf{v}_{\FF_{I,M}(Q),i,h}\in\prod_{j=1}^h[0,c_j]\}}{N_{I,M}(Q)}=\rho_{Q,I,M}\left(\bigcap_{j=1}^h(R_M\circ r_M^{j-1})^{-1}[0,c_j]\right).\]
For a given $j\in\N$, the function $R_M\circ r_M^j$, like $R_M$ itself, is piecewise rational and the domain on which $R_M\circ r_M^j$ is defined by a given rational function is a union of polygons. Indeed, we have seen in Section \ref{sec:5.1} that $\Omega_M$ can be divided into a countable number of polygons on each of which $r_M$ is linear, which implies that the same property holds for $r_M^j$ for any $j\in\N$. (This is in fact true for all $j\in\Z$ since $r_M$ is invertible.) This, together with the piecewise rationality of $R_M$ on polygons implies the same for $R_M\circ r_M^j$. As a consequence, the sets $(R_M\circ r_M^j)^{-1}[0,c_j]$ have boundaries of measure $0$. Hence, by our work in Section \ref{sec:6}, the limiting $h$-spacing measure $\nu_{I,M,h}$ of $(\FF_{I,M}(Q))$ exists and satisfies
\[\nu_{I,M,h}\left(\prod_{j=1}^h[0,c_j]\right)=\mu_{\Omega_M}\left(\bigcap_{j=1}^h(R_M\circ r_M^j)^{-1}\left(\left[0,\frac{\pi^2[\Gamma:H]}{3(\#M)}c_j\right]\right)\right).\]

Lastly, note that if $\frac{a}{q}<\frac{b}{p}$ are consecutive elements in $\FF_{I,M}(Q)$, $\frac{a'}{q'}$ succeeds $\frac{a}{q}$ in $\FF(Q)$, and $W_{H,Q}(\frac{b}{p})\in R_{c_1/c_3}$, then $p=-c_1q+c_3q'$ and
\[bq-ap=qp\left(\frac{b}{p}-\frac{a}{q}\right)=\frac{q(-c_1q+c_3q')}{Q^2}R_M\left(W_{H,Q}(\frac{a}{q})\right)=\frac{q(-c_1q+c_3q')}{Q^2}\frac{Q^2}{q(q'-\frac{c_1}{c_3}q)}=c_3.\]
Using this fact and Section \ref{sec:6}, one can show that for every $c_3\in\N$,
\[\lim_{Q\rar\infty}\frac{\#\{\frac{a}{q}<\frac{b}{q}\mbox{ consecutive in }\FF_{I,M}(Q):bq-ap=c_3\}}{N_{I,M}(Q)}=\mu_{\Omega_M}\left(\bigcup_{\substack{c_1/c_3\in\Q\\(c_1,c_3)=1}}R_{c_1/c_3}\right),\]
recovering a result Badziahin and Haynes proved for the sequence $(\FF_{Q,d})$ in \cite{BH}.

\section{The convergence $\rho_{Q,I,M}\rar\mu_{\Omega_M}$}\label{sec:6}

In this section, we prove the weak convergence $\rho_{Q,I,M}\rar\mu_{\Omega_M}$, and hence complete the proof of Theorem \ref{T1}. We first consider the measures $(\rho_{Q,I,M}^R)$ on $G/H$ defined by
\[d\rho_{Q,I,M}^R=\frac{N_{I,M}(Q)}{Q^2}\,d\rho_{Q,I,M}\,ds.\]
In other words, $\rho_{Q,I,M}^R$ is a measure concentrated on segments of the horocycle flow connecting $W_{H,Q}(\beta_i)$ to $W_{H,Q}(\beta_{i+1})$ for $0\leq i\leq N_{I,M}(Q)-1$. These segments connect to give one segment from $W_{H,Q}(\beta_0)$ to $W_{H,Q}(\beta_{N_{I,M}(Q)})$. So for a bounded, measurable function $f:G/H\rar\R$,
\begin{align}
\int f\,d\rho_{Q,I,M}^R=\frac{1}{Q^2}\int_{Q^2\beta_{0}}^{Q^2\beta_{N_{I,M}(Q)}}f\left(\left(\begin{array}{cc}1&0\\-s&1\end{array}\right)W_{H,Q}(0)\right)ds,\label{e61}
\end{align}
noting that $h_{Q^2\beta_i}W_{H,Q}(0)=W_{H,Q}(\beta_i)$. We wish to show that the sequence $(\rho_{Q,I,M}^R)$ converges weakly to $\frac{|I|\mu_{G/H}}{\mu_{G/H}(G/H)}$. Notice that
\[\begin{array}{c}
\left(\begin{array}{cc}1&0\\-s&1\end{array}\right)W_{H,Q}(0)=\left(\begin{array}{cc}0&1\\-1&0\end{array}\right)\left(\begin{array}{cc}Q&0\\0&Q^{-1}\end{array}\right)\left(\begin{array}{cc}1&\frac{s}{Q^2}\\0&1\end{array}\right)H.
\end{array}\]
So \eqref{e61} can be written as
\begin{align*}
\int f\,d\rho_{Q,I,M}^R=\int_{\beta_0}^{\beta_{N_{I,M}(Q)}}(\tilde{f}\circ g_Q)\left(\left(\begin{array}{cc}1&t\\0&1\end{array}\right)H\right)dt,
\end{align*}
where $\tilde{f}:G/H\rar\R$ is the composition of left multiplication on $G/H$ by $\left(\begin{smallmatrix}0&1\\-1&0\end{smallmatrix}\right)$ followed by $f$, and $g_Q:G/H\rar G/H$ is left multiplication by $\left(\begin{smallmatrix}Q&0\\0&Q^{-1}\end{smallmatrix}\right)$. Since $\bigcup_{Q\in\N}\FF_M(Q)$ is dense in $[0,1]$, $\beta_0\rar x_1$ and $\beta_{N_{I,M}(Q)}\rar x_2$ as $Q\rar\infty$. So if we define the measure $\rho_{Q,I,M}^{R'}$ on $G/H$ such that
\[\int f\,d\rho_{Q,I,M}^{R'}=\int_{x_1}^{x_2}(\tilde{f}\circ g_Q)\left(\left(\begin{array}{cc}1&t\\0&1\end{array}\right)H\right)dt,\]
for all bounded, measurable functions $f:G/H\rar\R$, then it is clear that $\rho_{Q,I,M}^R-\rho_{Q,I,M}^{R'}\rar0$ weakly. Thus to show that $\rho_{Q,I,M}^R\rar\frac{|I|\mu_{G/H}}{\mu_{G/H}(G/H)}$ weakly, it suffices to prove that $\rho_{Q,I,M}^{R'}\rar\frac{|I|\mu_{G/H}}{\mu_{G/H}(G/H)}$ weakly.

Like the convergence $\rho_{Q,I}^R\rar\mu_{G/\Gamma}$, this is a consequence of the equidistribution of closed horocycles in $G/H$. In particular, the argument for \cite[Theorem 7]{EM2} can be used to prove
\begin{equation}
\lim_{Q\rar\infty}\int_{x_1}^{x_2}(f\circ g_Q)\left(\left(\begin{array}{cc}1&t\\0&1\end{array}\right)H\right)dt=\frac{x_2-x_1}{\mu_{G/H}(G/H)}\int_{G/H}f\,d\mu_{G/H}\label{e62}
\end{equation}
for all functions $f:G/H\rar\R$ that are bounded and uniformly continuous.

Here we give an outline of this argument. Let $f:G/H\rar\R$ be bounded and uniformly continuous. Thicken the horocycle $C_{[x_1,x_2]}=\left\{\left(\begin{smallmatrix}1&t\\0&1\end{smallmatrix}\right)H:t\in[x_1,x_2]\right\}$ to an open set $U$ of points near $C_{[x_1,x_2]}$. By the uniform continuity of $f$, the integral of $f\circ g_Q$ over $C_{[x_1,x_2]}$ is, uniformly in $Q\in\N$, close to the integral of $f\circ g_Q$ over $U$, multiplied by $\frac{x_2-x_1}{\mu_{G/H}(U)}$ to account for the measure of the thickness of $U$. Then by the mixing of the geodesic flow $\{g_s:s>0\}$ in $G/H$ \cite{HM}, we have
\[\lim_{Q\rar\infty}\frac{x_2-x_1}{\mu_{G/H}(U)}\int_Uf\circ g_Q\,d\mu_{G/H}=\frac{x_2-x_1}{\mu_{G/H}(G/H)}\int_{G/H}f\,d\mu_{G/H}.\]
Letting $U$ shrink to $C_{[x_1,x_2]}$ then reveals \eqref{e62}.

Next, noting that $\mu_{G/H}$ is left $G$-invariant, we have
\begin{align*}
\lim_{Q\rar\infty}\int f\,d\rho_{Q,I,M}^{R'}&=\frac{|I|}{\mu_{G/H}(G/H)}\int_{G/H}\tilde{f}\,d\mu_{G/H}=\frac{|I|}{\mu_{G/H}(G/H)}\int_{G/H}f\,d\mu_{G/H},
\end{align*}
for every bounded, uniformly continuous function $f:G/H\rar\R$. By the Portmanteau theorem, this is equivalent to saying that $\rho_{Q,I,M}^{R'}\rar\frac{|I|\mu_{G/H}}{\mu_{G/H}(G/H)}$ weakly, which then implies that $\rho_{Q,I,M}^R\rar\frac{|I|\mu_{G/H}}{\mu_{G/H}(G/H)}$ weakly.

Our next step is to prove that if $\pi_M:G/H\rar\Omega_M$ is the projection $(x,s)\mapsto x$, where we are viewing $G/H$ as $\{(x,s)\in\Omega_M\times\R:0\leq s<R_M(x)\}$, then $\pi_{M*}\rho_{Q,I,M}^R\rar\frac{|I|\pi_{M*}\mu_{G/H}}{\mu_{G/H}(G/H)}$ weakly. So let $f\in C(\Omega_M)$ be nonnegative and bounded. For a given $\epsilon>0$, let $g_\epsilon:\Omega_M\rar\R$ be a continuous function such that $g_\epsilon\leq R_M$ and $\int_{\Omega_M}(R_M-g_\epsilon)\,d\mu_{\Omega''}<\frac{\epsilon}{2}$. Then \[O_\epsilon=\{h_sp_{a,b}m_iH:(a,b)\in\Omega^o,m_iH\in M,0<s<g_\epsilon(p_{a,b}m_iH)\}\] is an open subset of $G/H$ in which $\mu_{G/H}((G/H)\bsl O_\epsilon)<\frac{\epsilon}{2}$. So by the inner regularity of $\mu_{G/H}$ and Urysohn's lemma, there is a continuous function $\chi_\epsilon:G/H\rar[0,1]$ such that $\Supp\chi_\epsilon\subseteq O_\epsilon$ and $\chi_\epsilon^{-1}(\{1\})$ is a compact subset of $O_\epsilon$ with $\mu_{G/H}((G/H)\bsl\chi_\epsilon^{-1}(\{1\}))<\epsilon$.

Now notice that $\pi_M$ is continuous on $O_\epsilon$, and therefore $f_{\epsilon,1}=\chi_\epsilon\cdot(f\circ\pi_M),f_{\epsilon,2}=N-\chi_\epsilon\cdot(N-f\circ\pi_M)\in C(G/H)$, where $N>0$ is a constant such that $f\leq N$. Thus
\[\lim_{Q\rar\infty}\int_{G/H}f_{\epsilon,j}\,d\rho_{Q,I,M}^R=\frac{|I|}{\mu_{G/H}(G/H)}\int_{G/H}f_{\epsilon,j}\,d\mu_{G/H},\quad j=1,2.\]
Since $f_{\epsilon,1}\leq f\circ\pi_M\leq f_{\epsilon,2}$, we also have
\begin{align*}
\liminf_{Q\rar\infty}\int_{G/H}f\circ\pi_M\,d\rho_{Q,I,M}^R&\geq\frac{|I|}{\mu_{G/H}(G/H)}\int_{G/H}f_{\epsilon,1}\,d\mu_{G/H}\mbox{ and}\\
\limsup_{Q\rar\infty}\int_{G/H}f\circ\pi_M\,d\rho_{Q,I,M}^R&\leq\frac{|I|}{\mu_{G/H}(G/H)}\int_{G/H}f_{\epsilon,2}\,d\mu_{G/H}.
\end{align*}
By the properties of $\chi_\epsilon$,
\begin{align*}
\int_{G/H}f\circ\pi_M\,d\mu_{G/H}&\leq\int_{G/H}f_{\epsilon,1}\,d\mu_{G/H}+N\epsilon\mbox{ and}\\ \int_{G/H}f\circ\pi_M\,d\mu_{G/H}&\geq\int_{G/H}f_{\epsilon,2}\,d\mu_{G/H}-N\epsilon,
\end{align*}
and therefore the following two inequalities hold:
\begin{align*}
\liminf_{Q\rar\infty}\int_{G/H}f\circ\pi_M\,d\rho_{Q,I,M}^R&\geq\frac{|I|}{\mu_{G/H}(G/H)}\left(\int_{G/H}f\circ\pi_M\,d\mu_{G/H}-N\epsilon\right),\\
\limsup_{Q\rar\infty}\int_{G/H}f\circ\pi_M\,d\rho_{Q,I,M}^R&\leq\frac{|I|}{\mu_{G/H}(G/H)}\left(\int_{G/H}f\circ\pi_M\,d\mu_{G/H}+N\epsilon\right).
\end{align*}
Letting $\epsilon\rar0$ then yields \[\lim_{Q\rar\infty}\int_{G/H}f\circ\pi_M\,d\rho_{Q,I,M}^R=\frac{|I|}{\mu_{G/H}(G/H)}\int_{G/H}f\circ\pi_M\,d\mu_{G/H},\]
proving that $\pi_{M*}\rho_{Q,I,M}^R\rar\frac{|I|\pi_{M*}\mu_{G/H}}{\mu_{G/H}(G/H)}$ weakly.

As noted in Section \ref{sec:5}, $\frac{1}{R_M}$ can be well approximated in $L^1(\Omega_M,\mu_{\Omega_M})$ from above and below by continuous functions $g_{2,\epsilon}$ and $g_{3,\epsilon}$, and so one can easily show that $\frac{1}{R_M}\pi_{M*}\rho_{Q,I,M}^R\rar\frac{|I|\pi_{M*}\mu_{G/H}}{R_M\mu_{G/H}(G/H)}$ weakly using the fact that
\[g_{2,\epsilon}\pi_{M*}\rho_{Q,I,M}^R\rar\frac{g_{2,\epsilon}|I|\pi_{M*}\mu_{G/H}}{\mu_{G/H}(G/H)}\quad\mbox{and}\quad g_{3,\epsilon}\pi_{M*}\rho_{Q,I,M}^R\rar\frac{g_{3,\epsilon}|I|\pi_{M*}\mu_{G/H}}{\mu_{G/H}(G/H)}\] weakly. Notice that \[\frac{1}{R_M}\pi_{M*}\rho_{Q,I,M}^R=\frac{N_{I,M}(Q)}{Q^2}\rho_{Q,I,M}\quad\mbox{and}\quad \frac{|I|\pi_{M*}\mu_{G/H}}{R_M\mu_{G/H}(G/H)}=\frac{|I|(\#M)\mu_{\Omega_M}}{\mu_{G/H}(G/H)},\]
and hence $\frac{N_{I,M}(Q)}{Q^2}\rho_{Q,I,M}\rar\frac{|I|(\#M)\mu_{\Omega_M}}{\mu_{G/H}(G/H)}$ weakly. Since $\rho_{Q,I,M}$ is a probability measure for all $Q\in\N$, we have
\[\lim_{Q\rar\infty}\frac{N_{I,M}(Q)}{Q^2}=\lim_{Q\rar\infty}\frac{N_{I,M}(Q)}{Q^2}\rho_{Q,I,M}(\Omega_M)=\frac{|I|(\#M)\mu_{\Omega_M}(\Omega_M)}{\mu_{G/H}(G/H)}=\frac{|I|(\#M)}{\mu_{G/H}(G/H)},\]
implying that $N_{I,M}(Q)\sim\frac{|I|(\#M)Q^2}{\mu_{G/H}(G/H)}=\frac{3|I|(\#M)Q^2}{\pi^2[\Gamma:H]}$. This proves the equidistribution of $(\FF_M(Q))$ in $[0,1]$, and the weak convergence
\[\rho_{Q,I,M}\rar\mu_{\Omega_M},\]
completing the proof of Theorem \ref{T1}.

\section{The repulsion gap for Farey fractions $\frac{a}{q}$ such that $q\equiv1\bmod{m}$}\label{sec:7}

In this section, we determine the repulsion gap for Farey fractions with denominators congruent to $1$ modulo $m$. For a given increasing sequence $\mathcal{A}:=(A_n)$ of subsets of $[0,1]$ with limiting gap measure $\nu_\mathcal{A}$, we define the \textit{repulsion gap} of $\mathcal{A}$ to be \[K_\mathcal{A}:=\sup\{c\geq0:\nu_A([0,c])=0\}.\]
This means that if $\Delta_{av}(A_n)$ is the average gap between consecutive elements in $A_n$, then for a given $\epsilon\in(0,K_\mathcal{A})$, \[\lim_{n\rar\infty}\frac{\#\{x,x'\mbox{ consecutive in }A_n:x'-x\leq\epsilon\Delta_{av}(A_n)\}}{\#A_n-1}=0.\]
In other words, the proportion of the number of gaps of elements in $A_n$ that are smaller than $\epsilon\Delta_{av}(A_n)$ approaches $0$ as $n\rar\infty$. So $K_{\mathcal{A}}$ provides a measure for how big a large proportion of the gaps in $A_n$ must be for large $n$.

Let $I\subseteq[0,1]$ be a subinterval, $m\in\N$, and \[A=\{(a,1)\bmod{m}:a\in\{0,\ldots,m-1\}\}\subseteq(\Z/m\Z)^2\] so that $\FF_{I,m,A}(Q)$ is the set of fractions $\frac{a}{q}\in\FF(Q)\cap I$ with $q\equiv1\bmod{m}$. We now compute the repulsion gap for the sequence $(\FF_{I,m,A}(Q))$. First note that $\FF_{I,m,A}(Q)=\FF_{I,M}(Q)$, where $M$ is the set of cosets of the form
$\left(\begin{smallmatrix}a&b\\-1&d\end{smallmatrix}\right)\Gamma(m)$ in $\Gamma/\Gamma(m)$, where $a$, $b$, and $d$ are any integers such that $ad+b=1$. It is well known that $[\Gamma:\Gamma(m)]=m^3\prod_{p|m}(1-\frac{1}{p^2})$. Also, since the congruence $ad+b\equiv1\bmod{m}$ has $m^2$ solutions and each coset of $\Gamma/\Gamma(m)$ is completely determined by the congruence classes modulo $m$ of the entries of one of its elements, there are $m^2$ cosets in $M$. So by \eqref{e44}, the repulsion gap of $(\FF_{I,m,A}(Q))$ is \[\frac{3c'}{\pi^2m\prod_{p|m}(1-\frac{1}{p^2})}\]
where $c'=\sup\{c\geq0:\mu_{\Omega_M}(R_M^{-1}[0,c])=0\}$. We have previously found that for a given nonnegative

\noindent fraction $\frac{c_1}{c_3}$, $f_{c_1/c_3}^{-1}[0,c]$ is nonempty only if $c\geq\max\{c_1,c_3\}$. Also, $f_0^{-1}[0,c]$ is a subset of
\[\{p_{a,b}m_iH:b\geq\frac{1}{ca},i\in\{1,\ldots,k\}\}\]
which clearly has positive $\mu_{\Omega_M}$-measure if and only if $c>1$. Thus, we have $c'\geq1$. On the other hand, notice that if $m_1=\left(\begin{smallmatrix}1&0\\-1&1\end{smallmatrix}\right)$ and $m_2:=\left(\begin{smallmatrix}0&1\\-1&0\end{smallmatrix}\right)$, then $m_1\Gamma(m),m_2\Gamma(m)\in M$ and \[m_1m_2^{-1}=\left(\begin{array}{cc}0&-1\\1&1\end{array}\right)\in m_1\Gamma(m)m_2^{-1}.\]
This implies that $f_0^{-1}[0,c]$ contains the set $\{p_{a,b}m_1H:b\geq\frac{1}{ca}\}$, which has positive $\mu_{\Omega_M}$-measure when $c>1$. This proves that $c'=1$, and hence the repulsion gap, $K_{m,A}$, of $(\FF_{I,m,A}(Q))$ is given by \eqref{e12}.

Figure \ref{fig:3} depicts numerical approximations of densities of the revised measures $\nu_{I,m,A}'$ given by $\nu_{I,m,A}'[0,c]=\nu_{I,m,A}[0,K_{m,A}c]$ for $m=3,6,11$. The multiplication by $K_{m,A}$ makes $\nu_{I,m,A}'$ the limiting measure corresponding to \eqref{e43} in which the normalization of the gaps is $Q^2$, which allows for an even comparison of the sequences. The initial interval $[0,1]$ on which the densities are zero in Figure \ref{fig:3} reflect the fact that the constant $c'$ above equals $1$ for all three sequences.

\begin{figure}
\centering
\includegraphics[width=0.68\textwidth]{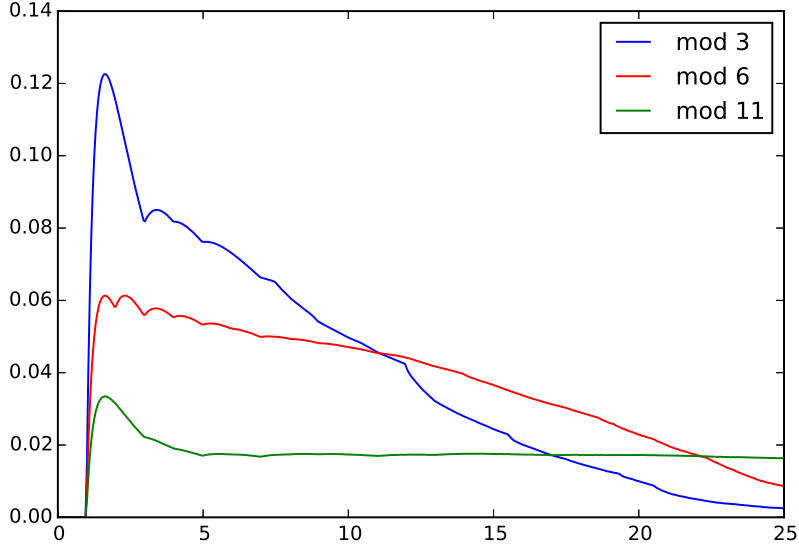}
\caption{Revised gap distribution densities for fractions with denominators congruent to 1 modulo 3, 6, and 11}
\label{fig:3}
\end{figure}

\section{Proof of Theorem \ref{T2}}\label{sec:8}

As in Sections \ref{sec:3}--\ref{sec:6}, we let $H\subseteq\Gamma$ be a finite index subgroup and $M=\{m_1H,\ldots,m_kH\}\subseteq\Gamma/H$ be nonempty and closed under left multiplication by $\left(\begin{smallmatrix}1&-1\\0&1\end{smallmatrix}\right)$, and $I\subseteq[0,1]$ be a subinterval. Recall that $S_{I,M}(n,\alpha,c)$ is the set of $\xi\in I$ for which there exists $\frac{a}{q}\in\FF_M(\lfloor nc\rfloor)$ such that $q\geq n$ and $|q\xi-a|\leq\frac{\alpha}{q}$; and that we aim to show that the limits of the sequences $(\lambda(S_M(n,\alpha,c)))$ and $(\lambda(S_{I,M}(n,\alpha,c)))$ exist, and
\[\lim_{n\rar\infty}S_{I,M}(n,\alpha,c)=|I|\varrho_M(\alpha,c),\quad\mbox{where}\quad\varrho_M(\alpha,c)=\lim_{n\rar\infty}\lambda(S_M(n,\alpha,c)).\]
The idea of the proof is to show that the measures $\lambda(S_{I,M}(n,\alpha,c))$ can, up to a small change, be written as sums of expressions of the form
\begin{equation}\label{e81}
\int_{\mathscr{H}(n/\lfloor nc\rfloor)}f\,d\rho_{\lfloor nc\rfloor,I,M},
\end{equation}
where $\mathscr{H}(\frac{n}{\lfloor nc\rfloor})\subseteq\Omega_M$ is a subset having a boundary of $\mu_{\Omega_M}$-measure $0$, and $f$ is a piecewise smooth function which is bounded on $\mathscr{H}(C)$ for any $C>0$. We then use the convergence $\rho_{Q,I,M}\rar\mu_{\Omega_M}$, in addition to the fact that the region $\mathscr{H}(\frac{n}{\lfloor nc\rfloor})$ becomes $\mathscr{H}(\frac{1}{c})$ as $n\rar\infty$, to show that the above integral approaches
\[\frac{3|I|(\#M)}{\pi^2[\Gamma:H]}\int_{\mathscr{H}(1/c)}f\,d\mu_{\Omega_M}.\]

We begin by following the process in \cite{XZ2} of using the inclusion-exclusion principle to rewrite $\lambda(S_{I,M}(n,\alpha,c))$ as close to a linear combination of measures of intervals. For $\alpha>0$, $c\geq1$, and $n\in\N$, let $Q=\lfloor nc\rfloor$ and 
\[\FF_{I,M}(Q)=\left\{\beta_0=\frac{a_0}{q_0}<\beta_1=\frac{a_1}{q_1}<\cdots<\beta_{N_{I,M}(Q)}=\frac{a_{N_{I,M}(Q)}}{q_{N_{I,M}(Q)}}\right\},\]
with $(a_i,q_i)=1$. Then for every $\beta_i\in\FF_{I,M}(Q)$, let 
\[J(\beta_i)=\bigcup_{i=0}^{N_{I,M}(Q)}\left[\frac{a_i}{q_i}-\frac{\alpha}{q_i^2},\frac{a_i}{q_i}+\frac{\alpha}{q_i^2}\right].\]
and then define
\[S_{I,M}'(n,\alpha,c)=\bigcup_{\substack{\beta_i\in\FF_{I,M}(Q)\\q_i\geq n}}J(\beta_i).\]
The only difference between $S_{I,M}'(n,\alpha,c)$ and $S_{I,M}(n,\alpha,c)$ is that the latter includes the union of the intervals of the form $I\cap[\frac{a}{q}-\frac{\alpha}{q^2},\frac{a}{q}+\frac{\alpha}{q^2}]$, where $\frac{a}{q}$ is a fraction in $\FF_M(Q)$ with $q\geq n$ which lies outside of $I$. It is clear that the measure of the union of these intervals cannot exceed $\frac{2\alpha}{n^2}$, and so the sequences $(\lambda(S_{I,M}(n,\alpha,c)))$ and $(\lambda(S_{I,M}'(n,\alpha,c)))$ converge and have the same limit if one converges. Thus from now on, we examine the sets $S_{I,M}'(n,\alpha,c)$.

Now by the inclusion-exclusion principle, we have
\begin{align}
\lambda(S_{I,M}'(n,\alpha,c))&=\lambda\left(\bigcup_{\substack{\beta_i\in\FF_{I,M}(Q)\\q_i\geq n}}\lambda(J(\beta_i))\right)\nonumber\\
&=\sum_{r=0}^{N_{I,M}(Q)-1}(-1)^r\sum_{0=j_0<\cdots<j_r\leq N_{I,M}(Q)}\sum_{\substack{i=0\\q_{i+j_s}\geq n,0\leq s\leq r}}^{N_{I,M}(Q)-j_r}\lambda\left(\bigcap_{s=0}^rJ(\beta_{i+j_s})\right).\label{e82}
\end{align}
By \cite[Lemma 3]{XZ2}, there exists an integer $K$, depending only on $\alpha$ and $c$, such that if $\frac{a}{q},\frac{a'}{q'}\in\FF(Q)$ such that $q,q'\geq n$ and $J(\frac{a}{q})\cap J(\frac{a'}{q'})\neq\emptyset$, then there are at most $K-1$ elements in $\FF(Q)$ between $\frac{a}{q}$ and $\frac{a'}{q'}$. It follows that if $\beta_i,\beta_j\in\FF_M(Q)$ with $q_i,q_j\geq n $ and $J(\beta_i)\cap J(\beta_j)\neq\emptyset$, then $|i-j|\leq K$. We can thus rewrite \eqref{e82} as
\begin{equation}
\sum_{r=0}^K(-1)^r\sum_{0=j_0<\cdots<j_r\leq K}\sum_{\substack{i=1\\q_{i+j_s}\geq n,0\leq s\leq r}}^{N_{I,M}(Q)-j_r}\lambda\left(\bigcap_{s=0}^rJ(\beta_{i+j_s})\right).\label{e83}
\end{equation}

Next, again analogous to \cite{XZ2}, we construct a region $\mathscr{H}_{j_1,\ldots,j_r}(\frac{n}{Q})$ in $\Omega_M$ having the property that $W_{H,Q}(\beta_i)\in\mathscr{H}_{j_1,\ldots,j_r}(\frac{n}{Q})$ if and only if $q_{i+j_s}\geq n$ for $0\leq s\leq r$, and then write
\[\lambda\left(\bigcap_{s=0}^r J(\beta_{i+j_s})\right)\]
as a piecewise smooth function of $W_{H,Q}(\beta_i)$. In this way, we will rewrite \eqref{e83}, up to a small change, as a linear combination of expressions  in the form \eqref{e81} as mentioned above.

The fraction $\beta_i\in\FF_{I,M}(Q)$ satisfies $q_{i+j_s}\geq n$ if and only if
\[W_{H,Q}(\beta_{i+j_s})=r_M^{j_s}(W_{H,Q}(\beta_i))\in\left\{p_{a,b}m_iH\in\Omega_M:a\geq\frac{n}{Q}\right\}.\]
So the set of $\beta_i$ such that $q_{i+j_s}\geq n$ for $0\leq s\leq r$ are such that $W_{H,Q}(\beta_i)\in \mathscr{H}_{j_1,\ldots,j_r}(\frac{n}{Q})$, where for $t\in(0,1]$,
\[\mathscr{H}_{j_1,\ldots,j_r}(t)=\bigcap_{s=0}^rr_M^{-j_s}\{p_{a,b}m_iH\in\Omega_M:a\geq t\}.\]
Note that $\mathscr{H}_{j_1,\ldots,j_r}(t)$ is a countable union of polygons. This follows easily from our observation in Section \ref{sec:5.1} that for a given $j\in\Z$, $\Omega_M$ can be divided into a countable number of polygones $P'$ such that $r_M^j$ is linear on $P'$. Hence any set of the form $r_M^{-j}\{p_{a,b}m_iH\in\Omega_M:a\geq t\}$ is a countable union of polygons, and thus $\mathscr{H}_{j_1,\ldots,j_r}(t)$ is as well.

Next, for $j\in\N$, let $R_M^{(j)}:\Omega_M\rar\R$ be the $j$th return time function defined by
\[R_M^{(j)}=\sum_{i=0}^{j-1}R_M\circ r_M^i.\]
(Let $R_M^{(0)}\equiv0$.) We then have
\[R_M^{(j)}(W_{H,Q}(\beta_i))=Q^2(\beta_{i+j}-\beta_i).\]
Also, define the function $p:\Omega_M\rar\R$ by $p(p_{a,b}m_iH)=a$.
We then have
\begin{align*}
\lambda\left(\bigcap_{s=0}^rJ(\beta_{i+j_s})\right)&=\max\left\{0,\min_{0\leq s\leq s'\leq r}\left\{\left(\frac{a_{i+j_s}}{q_{i+j_s}}+\frac{\alpha}{q_{i+j_s}^2}\right)-\left(\frac{a_{i+j_{s'}}}{q_{i+j_{s'}}}-\frac{\alpha}{q_{i+j_{s'}}^2}\right)\right\}\right\}\\
&=\frac{1}{Q^2}\max\left\{0,\min_{0\leq s\leq s'\leq r}\left\{\alpha\left(\frac{Q^2}{q_{i+j_s}^2}+\frac{Q^2}{q_{i+j_{s'}}^2}\right)-Q^2\left(\frac{a_{i+j_{s'}}}{q_{i+j_{s'}}}-\frac{a_{i+j_s}}{q_{i+j_s}}\right)\right\}\right\}\\
&=\frac{1}{Q^2}\max\Big\{0,\min_{0\leq s\leq s'\leq r}\big\{\alpha\big((p\circ r_M^{j_s})(W_{H,Q}(\beta_i))^{-2}+(p\circ r_M^{j_{s'}})(W_{H,Q}(\beta_i))^{-2}\big)\\ &\qquad\qquad\qquad\qquad\qquad\quad-(R_M^{(j_{s'}-j_s)}\circ r_M^{j_s})(W_{H,Q}(\beta_i)))\big\}\Big\}\\
&=\frac{1}{Q^2}f_{j_1,\ldots,j_r}^{(\alpha)}(W_{H,Q}(\beta_i)),
\end{align*}
where $f_{j_1,\ldots,j_r}^{(\alpha)}:\Omega_M\rar\R$ is given by
\[f_{j_1,\ldots,j_r}^{(\alpha)}=\max\Big\{0,\min_{0\leq s\leq s'\leq r}\big\{\alpha\big((p\circ r_M^{j_s})^{-2}+(p\circ r_M^{j_{s'}})^{-2}\big)-(R_M^{(j_{s'}-j_s)}\circ r_M^{j_s})\big\}\Big\}.\]

We can now rewrite \eqref{e83} as
\begin{align*}
&\sum_{r=0}^K(-1)^r\sum_{0=j_0<\cdots<j_r\leq K}\frac{1}{Q^2}\sum_{\substack{\beta_i\in\FF_{I,M}(Q),\\i\leq N_{I,M}(Q)-j_r}}f_{j_1,\ldots,j_r}^{(\alpha)}(W_{H,Q}(\beta_i))\\
&\qquad\qquad\qquad=\sum_{r=0}^K(-1)^r\sum_{0=j_0<\cdots<j_r\leq K}\Biggl(\frac{N_{I,M}(Q)}{Q^2}\int_{\mathscr{H}_{j_1,\ldots,j_r}(n/Q)}f_{j_1,\ldots,j_r}^{(\alpha)}\,d\rho_{Q,I,M}\\
&\qquad\qquad\qquad\hspace{130pt}+O\biggl(\frac{K\|f_{j_1,\ldots,j_r}^{(\alpha)}|_{\mathscr{H}_{j_1,\ldots,j_r}(n/Q)}\|_\infty}{Q^2}\biggr)\Biggr),
\end{align*}
where, for $C>0$,
\begin{align*}
\|f_{j_1,\ldots,j_r}^{(\alpha)}|_{\mathscr{H}_{j_1,\ldots,j_r}(C)}\|_\infty&=\sup\left\{|f_{j_1,\ldots,f_r}^{(\alpha)}(p_{a,b}m_iH)|:p_{a,b}m_iH\in\mathscr{H}_{j_1,\ldots,j_r}(C)\right\}\\
&\leq\frac{2\alpha}{C^2}. %\frac{2\alpha Q^2}{n^2}\leq2\alpha c^2.
\end{align*}
With $C=\frac{n}{Q}$, we have $\frac{2\alpha}{C^2}\leq2\alpha c^2$. Thus the big $O$ term above is negligible, and to complete the proof, it remains to show the existence of
\[\lim_{n\rar\infty}\int_{\mathscr{H}_{j_1,\ldots,j_r}(n/Q)}f_{j_1,\ldots,j_r}^{(\alpha)}\,d\rho_{Q,I,M}\]
for all $j_1,\ldots,j_r$.

Now by the properties of $R_M$, $r_M$, and $p$, it is clear that $f_{j_1,\ldots,j_r}^{(\alpha)}$ is a piecewise smooth function. We have proven above that for a fixed $C>0$, $f_{j_1,\ldots,j_r}^{(\alpha)}$ is bounded on $\mathscr{H}_{j_1,\ldots,j_r}(C)$. Also, it is clear that $\mathscr{H}_{j_1,\ldots,j_r}(C)$ has a boundary of $\mu_{\Omega_M}$-measure $0$, implying that
\[\lim_{n\rar\infty}\int_{\mathscr{H}_{j_1,\ldots,j_r}(C)}f_{j_1,\ldots,j_r}^{(\alpha)}\,d\rho_{Q,I,M}=\int_{\mathscr{H}_{j_1,\ldots,j_r}(C)}f_{j_1,\ldots,j_r}^{(\alpha)}\,d\mu_{\Omega_M}.\]
Since $\frac{n}{Q}\geq\frac{1}{c}$, and thus $\mathscr{H}_{j_1,\ldots,j_r}(\frac{n}{Q})\subseteq\mathscr{H}_{j_1,\ldots,j_r}(\frac{1}{c})$, we have
\begin{align*}
\limsup_{n\rar\infty}\int_{\mathscr{H}_{j_1,\ldots,j_r}(n/Q)}f_{j_1,\ldots,j_r}^{(\alpha)}\,d\rho_{Q,I,M}&\leq\lim_{n\rar\infty}\int_{\mathscr{H}_{j_1,\ldots,j_r}(1/c)}f_{j_1,\ldots,j_r}^{(\alpha)}\,d\rho_{Q,I,M}\\
&=\int_{\mathscr{H}_{j_1,\ldots,j_r}(1/c)}f_{j_1,\ldots,j_r}^{(\alpha)}\,d\mu_{\Omega_M}.
\end{align*}
On the other hand, for a given $\epsilon>0$, we have $\frac{n}{Q}\leq\frac{1}{c}+\epsilon$ for large $n$. Therefore,
\begin{align*}
\liminf_{n\rar\infty}\int_{\mathscr{H}_{j_1,\ldots,j_r}(n/Q)}f_{j_1,\ldots,j_r}^{(\alpha)}\,d\rho_{Q,I,M}&\geq\lim_{n\rar\infty}\int_{\mathscr{H}_{j_1,\ldots,j_r}(1/c+\epsilon)}f_{j_1,\ldots,j_r}^{(\alpha)}\,d\rho_{Q,I,M}\\
&=\int_{\mathscr{H}_{j_1,\ldots,j_r}(1/c+\epsilon)}f_{j_1,\ldots,j_r}^{(\alpha)}\,d\mu_{\Omega_M}.
\end{align*}
By the continuity of measure from below, letting $\epsilon\rar0$ yields
\[\liminf_{n\rar\infty}\int_{\mathscr{H}_{j_1,\ldots,j_r}(n/Q)}f_{j_1,\ldots,j_r}^{(\alpha)}\,d\rho_{Q,I,M}\geq\int_{\mathscr{H}_{j_1,\ldots,j_r}'(1/c)}f_{j_1,\ldots,j_r}^{(\alpha)}\,d\mu_{\Omega_M},\]
where
\[\mathscr{H}_{j_1,\ldots,j_r}'\left(\frac{1}{c}\right)=\bigcap_{s=0}^rr_M^{-j_s}\left\{p_{a,b}m_iH\in\Omega_M:a>\frac{1}{c}\right\}.\]
We clearly have
\[\mu_{\Omega_M}\left(\bigcap_{s=0}^rr_M^{-j_s}\left\{p_{a,b}m_iH\in\Omega_M:a=\frac{1}{c}\right\}\right)=0,\]
and thus
\[\lim_{n\rar\infty}\int_{\mathscr{H}_{j_1,\ldots,j_r}(n/Q)}f_{j_1,\ldots,j_r}^{(\alpha)}\,d\rho_{Q,I,M}=\int_{\mathscr{H}_{j_1,\ldots,j_r}(1/c)}f_{j_1,\ldots,j_r}^{(\alpha)}\,d\mu_{\Omega_M}.\]
Noting again that $N_{I,M}(Q)\sim\frac{3|I|(\#M)Q^2}{\pi^2[\Gamma:H]}$, we have completed the proof of Theorem \ref{T2}, with
\[\varrho_M(\alpha,c)=\frac{3(\#M)}{\pi^2[\Gamma:H]}\sum_{r=0}^K(-1)^r\sum_{0=j_0<\cdots<j_r\leq K}\int_{\mathscr{H}_{j_1,\ldots,j_r}(1/c)}f_{j_1,\ldots,j_r}^{(\alpha)}\,d\mu_{\Omega_M}.\]

\ 

\noindent\textbf{Acknowledgements}. I thank my advisor Florin Boca for his guidance in this research. I thank Jayadev Athreya and Jens Marklof for constructive comments on the first draft of this paper, and I thank the referee for helpful comments and suggestions. I also acknowledge support from Department of Education Grant P200A090062, \lq\lq University of Illinois GAANN Mathematics Fellowship Project.\rq\rq

\end{document}